\documentclass{article}
\usepackage{amsmath}
\usepackage{amssymb}
\usepackage{amsthm}
\usepackage{natbib}
\usepackage[left=2cm,right=2cm,top=2cm,bottom=2cm]{geometry}
\usepackage{hyperref}
\usepackage{url}
\usepackage{placeins}
\usepackage{morefloats}
\usepackage[ruled,noend]{algorithm2e}
\usepackage{pgf,tikz}
\usepackage{mathtools}
\usepackage{authblk}
\usepackage{graphicx}
\usepackage[lofdepth, lotdepth]{subfig}

\usetikzlibrary{calc, positioning}
\newtheorem{theorem}{Theorem}

\newtheorem{definition}{Definition}
\newtheorem{example}{Example}
\newtheorem{lemma}{Lemma}
\newtheorem{proposition}{Proposition}
\DeclareMathOperator{\pa}{pa}

\DeclareMathOperator{\ch}{ch}
\DeclareMathOperator{\an}{an}
\DeclareMathOperator{\de}{de}
\DeclareMathOperator{\cc}{cc}
\DeclareMathOperator{\cn}{cn}
\DeclareMathOperator{\nd}{nd}
\DeclareMathOperator{\nb}{nb}
\DeclareMathOperator{\forb}{forb}
\DeclareMathOperator{\ignore}{ignore}

\DeclareMathOperator{\false}{\textup{\texttt{false}}}
\DeclareMathOperator{\true}{\textup{\texttt{true}}}
\DeclareMathOperator{\out}{\textup{\texttt{out}}}
\DeclareMathOperator{\visited}{\textup{\texttt{visited}}}
\DeclareMathOperator{\stack}{\textup{\texttt{stack}}}

\newcommand{\ort}{\perp \!\!\!\perp}

\newcommand{\q}{\pi(A \mid \mathbf{L})}
\newcommand{\qa}{\pi(a \mid \mathbf{L})}
\newcommand{\qap}{\pi(a^{\prime} \mid \mathbf{L})}

\begin{document}

\title{Efficient adjustment sets in causal graphical models with hidden
variables}
\author[1]{Ezequiel Smucler \thanks{esmucler@utdt.edu}}
\affil[1]{Department of Mathematics and Statistics, Universidad Torcuato Di Tella}
\author[2]{Facundo Sapienza \thanks{fsapienza@berkeley.edu}}
\affil[3]{Department of Statistics, University of California, Berkeley}
\author[3]{Andrea Rotnitzky \thanks{arotnitzky@utdt.edu}}
\affil[3]{Department of Economics, Universidad Torcuato Di Tella, CONICET and Department of Biostatistics, Harvard T.H. Chan School of Public Health}
\maketitle

\begin{abstract}
We study the selection of covariate adjustment sets for estimating the value
of point exposure dynamic policies, also known as dynamic treatment
regimes, assuming a non-parametric causal graphical model with hidden
variables, in which at least one adjustment set is fully observable. We show
that recently developed criteria, for graphs without hidden variables, to compare the asymptotic variance of non-parametric estimators of static policy values that control for certain adjustment sets, are also valid under dynamic policies and graphs with hidden variables. We show that there exist adjustment sets that are optimal minimal (minimum), in the
sense of yielding estimators with the smallest variance among those that
control for adjustment sets that are minimal (of minimum cardinality).
Moreover, we show that if either no variables are hidden or if all the
observable variables are ancestors of either treatment, outcome, or the
variables that are used to decide treatment, a globally optimal adjustment
set exists. We provide polynomial time algorithms to compute the globally
optimal (when it exists), optimal minimal, and optimal minimum adjustment
sets. Our results are based on the construction of an undirected graph in
which vertex cuts between the treatment and outcome variables correspond to
adjustment sets. In this undirected graph, a partial order between minimal
vertex cuts can be defined that makes the set of minimal cuts a lattice.
This partial order corresponds directly to the ordering of the asymptotic
variances of the corresponding non-parametrically adjusted estimators.
\end{abstract}

\affil[1]{Department of Mathematics and Statistics, Universidad Torcuato Di
Tella}

\affil[3]{Department of Statistics, University of California, Berkeley}

\affil[3]{Department of Economics, Universidad Torcuato Di Tella, CONICET
and Department of Biostatistics, Harvard T.H. Chan School of Public Health}

\section{Introduction}

In this paper we consider the selection of covariate adjustment variables
for off-policy evaluation \citep{precup} in single time contextual decision
making problems. Specifically, we consider the choice of variables that
suffice for estimating the value of a point exposure contextual policy by
the method of covariate adjustment, when the available data come from a
different policy. We assume a causal graphical model with, possibly, hidden
variables in which at least one valid adjustment set is fully observable.
The value of a policy, also known as the interventional mean, is defined as
the mean of an outcome (reward) under the policy. In the statistics
literature, a policy is referred to as a dynamic treatment regime %
\citep{Robins93, murphy2001marginal, Robins2004, schulte}.

A practical application of the methods described in this paper is in the
design of planned observational studies. Investigators designing such study
might use the existing graphical criteria for identifying the class of
candidate valid covariate adjustment sets 
\citep{pearl-causality,
kuroki-miyawa, shpitser-adjustment}, and then apply the methods described in
this paper to select an adjustment set that satisfies one of three
optimality criteria that we consider here. Each criterion is defined by
selecting the observable adjustment set that yields the non-parametrically
adjusted estimator with smallest asymptotic variance among those that
control for observable adjustment sets in a given class, specifically the
class of (i) all adjustment sets, (ii) all minimal adjustment sets, or (iii)
all adjustment sets that have minimum cardinality. We refer to adjustment
sets satisfying criterion (i) as globally optimal, (ii) as optimal minimal
and (iii) as optimal minimum. A minimal adjustment set is such that removal
of any variable from it results in an invalid adjustment set.

Our proposal extends existing methods for selecting covariate adjustment
sets in a number of ways. Specifically, \cite{kuroki-miyawa} proposed a
graphical criteria for comparing the asymptotic variance of estimators of
the value of a point exposure policy that control for two different
adjustment sets, under the following assumptions: (i) a linear causal
graphical model with no hidden variables, (ii) a policy that is an affine
function of a single covariate $L$, (iii) estimators obtained by ordinary
least squares, and (iv) adjustment sets that consist only on $L$ and another
single variable. More recently, assuming (i) and (iii) as in \cite%
{kuroki-miyawa}, but restricting attention to static policies, i.e. those
that do not depend on covariates, \cite{perkovic} derived a general
graphical characterization of the globally optimal adjustment set. \cite%
{didelez} derived an alternative graphical characterization of this set. 
\cite{perkovic} additionally provided a criterion for comparing adjustment
sets that is more widely applicable than earlier existing criteria %
\citep{kuroki-cai, kuroki-miyawa}. They also showed, by means of a
counter-example, that in graphs with hidden variables with observable
adjustment sets there may not exist a globally optimal adjustment set. \cite%
{eff-adj} extended the results of \cite{perkovic} to non-parametric causal
graphical models and non-parametrically adjusted estimators. Moreover, they
provided a graphical characterization of the optimal minimal adjustment set.

The contributions of this paper are:

\begin{enumerate}
\item We show that the criteria of \cite{perkovic} for comparing certain
pairs of adjustment sets remains valid for point exposure dynamic treatment
regimes in non-parametric causal graphical models with hidden variables.

\item We show that if either no variables are hidden, or if all the
observable variables are ancestors of either treatment, outcome, or the
variables that are used to decide treatment, a globally optimal adjustment
set exists, and we provide a graphical characterization of it.

\item We show that in graphs with hidden variables that admit at least one
observable adjustment set there always exist optimal minimal and optimal
minimum adjustment sets and we provide graphical characterizations of them.

\item We provide polynomial time algorithms to compute the globally optimal
(when it exists), optimal minimal, and optimal minimum adjustment sets.
\end{enumerate}

The formulation of our computational algorithms builds on previous work
regarding graphical characterizations of adjustment sets and algorithms to
compute them developed in \cite{acid}, \cite{pearl-tian}, \cite{textor12}, 
\cite{van14} and \cite{vanAI}. Specifically, \cite{acid} and \cite%
{pearl-tian} proposed polynomial time algorithms for finding minimal
d-separators and minimum size d-separators for a given pair of vertices on a
directed acyclic graph. Building on the work of these authors, \cite{vanAI}
(see also the earlier papers \cite{textor12, van14}) provided a constructive
graphical characterization of adjustments sets for static interventions, and
also polynomial time algorithms to compute adjustment sets, minimal
adjustment sets, and minimum size adjustments, possibly under constraints on
which some variables must be included and others must not be included in the
desired adjustment set.

The rest of the paper is organized as follows. In Section \ref{sec:back} we
review basic results and definitions regarding causal graphical models.
Next, in Section \ref{sec:adj}, we extend the definition of adjustment sets
of \cite{shpitser-adjustment} and \cite{maathuis2015} to dynamic treatment
regimes. In Section \ref{sec:nonparam} we review non-parametric estimation
of policy values. Section \ref{sec:compare} extends the criteria of \cite%
{eff-adj} to compare certain adjustment sets to dynamic interventions in
graphs with hidden variables. In Section \ref{sec:new_graph} we provide
graphical characterizations of globally optimal (when it exists), optimal
minimal, and optimal minimum adjustment sets and in Section \ref{sec:algos}
we provide polynomial time algorithms to compute them. Finally in Section %
\ref{sec:discussion} we conclude with a discussion of open problems. The
proofs of all the results stated in the main paper are available in the
Supplementary Material.

\section{Background}

\label{sec:back} We now review some basic definitions and results of the
theory of graphical models.

\subsection{Definitions and notation}

\label{sec:def}

\subsubsection{Undirected graphs}

\textbf{Undirected graphs}. An undirected graph $\mathcal{H}=(\mathbf{V},%
\mathbf{E})$ consists of a finite vertex set $\mathbf{V}$ and a set
undirected edges $\mathbf{E}$. An undirected edge between two vertices $V$, $%
W$ is represented by $V - W$. Given a set of vertices $\mathbf{Z}\subset 
\mathbf{V}$ the induced subgraph $\mathcal{H}_{\mathbf{Z}}=(\mathbf{Z},%
\mathbf{E}_{Z})$ is the graph obtained by considering only vertices in $%
\mathbf{Z}$ and edges between vertices in $\mathbf{Z}$. We will sometimes
use the notation $\mathbf{V}(\mathcal{H})$ and $\mathbf{E}(\mathcal{H})$ to
refer to the vertex and edge sets respectively of an undirected graph $%
\mathcal{H}$.

\noindent\textbf{Paths}. If $V - W$ is an edge in $\mathcal{H}$ then we say
that $V$ and $W$ are adjacent. A path from a vertex $V$ to a vertex $W$ in
graph $\mathcal{H}$ is a sequence of vertices $(V_1, \dots, V_{j})$ such
that $V_{1}=V $, $V_{j}=W$ and $V_i$ and $V_{i+1}$ are adjacent in $\mathcal{%
H}$ for all $i \in \{1, \dots, j-1\}$. We define the set of neighbors of a
vertex $V$ as the set of vertices adjacent to $V$ and use the notation $\nb_{%
\mathcal{H}}(V)$ for this set. For a set of vertices $\mathbf{Z}\subset 
\mathbf{V}$ we let $\nb_{\mathcal{H}}(\mathbf{Z}) \equiv \cup_{Z\in\mathbf{Z}%
} \nb_{\mathcal{H}}(Z)$.

\noindent\textbf{Connected components}. If there exists a path from a vertex 
$V $ to a vertex $W$ in graph $\mathcal{H}$ we say that $V$ and $W$ are
connected. A connected component of $\mathcal{H}$ is a maximal subset of
vertices $\mathbf{U}$ such that for all $V,W \in\mathbf{U}$, $V$ and $W$ are
connected in $\mathcal{H}$ by a path that goes only through vertices in $%
\mathbf{U}$. For $\mathcal{H}=(\mathbf{V}, \mathbf{E})$ and $\mathbf{U}%
\subset \mathbf{V}$, $\partial_{\mathcal{H}} \mathbf{U}$ denotes the set of
vertices in $\mathbf{V}\setminus \mathbf{U}$ which are adjacent to at least
one vertex in $\mathbf{U}$. If moreover $Y \in \mathbf{V}\setminus \mathbf{U}
$, then $\cc(\mathbf{U},Y,\mathcal{H})$ will denote the connected component
of $\mathcal{H}_{\mathbf{V}\setminus \mathbf{U}}$ which contains $Y$.

\noindent \textbf{Vertex cuts}. Consider an undirected graph $\mathcal{H}=(%
\mathbf{V}, \mathbf{E})$. Let $A, Y \in \mathbf{V}$ and $\mathbf{Z} \subset 
\mathbf{V}$ such that $\mathbf{Z}\cap \lbrace A,Y\rbrace = \emptyset$. We
say that $\mathbf{Z} $ is an $A-Y$ cut or that $A$ and $Y$ are separated by $%
\mathbf{Z}$ if all paths between $A$ and $Y$ intersect a vertex in $\mathbf{Z%
} $. If $\mathbf{Z} $ is an $A-Y$ cut we write $A \perp_{\mathcal{H}} Y \mid 
\mathbf{Z}$. $\mathbf{Z} $ is a minimal $A-Y$ cut if it is an $A-Y$ cut and
no proper subset of $\mathbf{Z}$ is an $A-Y$ cut. $\mathbf{Z} $ is a minimum 
$A-Y$ cut if it is an $A-Y$ cut and there exists no $A-Y$ cut with a
cardinality smaller than the cardinality of $\mathbf{Z}$. 
Note that every minimum $A-Y$ cut is also minimal, but the reciprocal is
false in general.

\noindent \textbf{Lattices}. A lattice $\mathcal{L}=(L, \unlhd)$ is a set $L$
together with a relation $\unlhd$ that is reflexive, anti-symmetric and
transitive such that for all $a,b \in L$ there exists a greatest lower bound
for $a,b$ in $L$, called the inf of $a,b$, and a smallest upper bound for $%
a,b$ in $L$, called the sup of $a,b$.

\subsubsection{Directed graphs}

\textbf{Directed graphs}. A directed graph $\mathcal{G}=(\mathbf{V},\mathbf{E%
})$ consists of a finite vertex (also called node) set $\mathbf{V}$ and a
set of directed edges $\mathbf{E} \subset \mathbf{V}\times\mathbf{V}$. We
represent a directed edge between two vertices $V$, $W$ by $V\rightarrow W$.
Given a set of vertices $\mathbf{Z}\subset \mathbf{V}$ the induced subgraph $%
\mathcal{G}_{\mathbf{Z}}=(\mathbf{Z},\mathbf{E}_{Z})$ is defined as the
graph obtained by considering only vertices in $\mathbf{Z}$ and edges
between vertices in $\mathbf{Z}$. We will sometimes use the notation $%
\mathbf{V}(\mathcal{G})$ and $\mathbf{E}(\mathcal{G})$ to refer to the
vertex and edge sets respectively of the directed graph $\mathcal{G}$.

\noindent\textbf{Paths}. We say that two vertices are adjacent if there is
an edge between them. A path from a vertex $V$ to a vertex $W$ in graph $%
\mathcal{G} $ is a sequence of vertices $(V_1, \dots, V_{j})$ such that $%
V_{1}=V$, $V_{j}=W$ and $V_i$ and $V_{i+1}$ are adjacent in $\mathcal{G}$
for all $i \in \{1, \dots, j-1\}$. $V$ and $W$ are the endpoints of the
path. A path $(V_1, \dots, V_{j})$ is called directed or causal if $V_{i}\to
V_{i+1}$ for all $i \in \lbrace 1,\dots,j-1\rbrace$.

\noindent \textbf{Ancestry}. If $V\rightarrow W$, then $V$ is a parent of $W$
and $W$ is a child of $V$. If there is a directed path from $V$ to $W$, then 
$V$ is an ancestor of $W$ and $W$ a descendant of $V$. We follow the
convention that very vertex is an ancestor and a descendant of itself. The
sets of parents, children, ancestors and descendants of $V$ in $\mathcal{G}$
are denoted by $\pa_{\mathcal{G}}(V)$, $\ch_{\mathcal{G}}(V)$, $\an_{%
\mathcal{G}}(V)$ and $\de_{\mathcal{G}}(V)$ respectively. The set of
non-descendants of a vertex $V$ is defined as $\nd_{\mathcal{G}}(V)\equiv 
\mathbf{V}\setminus \de_{\mathcal{G}}(V) $. For a set of vertices $\mathbf{Z}
$ we define $\an_{\mathcal{G}}(\mathbf{Z})=\cup_{Z\in\mathbf{Z}} \an_{%
\mathcal{G}}(Z) $.

\noindent \textbf{Colliders and forks}. If $\delta$ is a path on a directed
graph $\mathcal{G}$, a vertex $V$ on $\delta$ is a collider on that path if $%
\delta$ contains a subpath $(U,V,W)$ such that $U\rightarrow V\leftarrow W$.
A vertex $V$ is a fork on the path $\delta $ if $\delta $ contains a subpath 
$(U,V,W)$ such that $U\leftarrow V\rightarrow W$.

\noindent \textbf{Directed cycles, DAGs}. A directed cycle is a directed
path from $V$ to $W$, together with the edge $W\rightarrow V$. A directed
graph without directed cycles is called a directed acyclic graph (DAG).

\noindent \textbf{d-separation} \citep{verma-pearl}. Let $\mathcal{G}$ be a
DAG with vertex set $\mathbf{V}$. Let $\mathbf{U},\mathbf{W},\mathbf{Z}$ be
distinct subsets of $\mathbf{V}$. A path $\delta$ in $\mathcal{G}$ between $%
U\in \mathbf{U}$ and $W\in \mathbf{W}$ is blocked by $\mathbf{Z}$ if at
least one of the following holds:

\begin{itemize}
\item[1.] There exists a vertex on $\delta$ that is not a collider and is an
element of $\mathbf{Z}$, or

\item[2.] There exists a vertex $C$ that is a collider on $\delta $ such
that neither $C$ nor its descendants are elements of $\mathbf{Z}$.
\end{itemize}

Two sets of vertices $\mathbf{U},\mathbf{W}$ are d-separated by $\mathbf{Z}$
in $\mathcal{G}$ if for any $U\in \mathbf{U}$ and $W\in \mathbf{W,}$ all
paths between $U$ and $W $ are blocked by $\mathbf{Z}$. If $\mathbf{U},%
\mathbf{W}$ are d-separated by $\mathbf{Z}$ we write $\mathbf{U} \perp
\!\!\!\perp_{\mathcal{G}} \mathbf{W} \mid \mathbf{Z}$.

\noindent \textbf{Moral graph.} Given a DAG $\mathcal{G}$ with vertex set $%
\mathbf{V}$, the associated moral graph $\mathcal{G}^{m}$ is an undirected
graph with the same vertex set as $\mathcal{G}$ and an edge $U- V$ if any of
the following hold in $\mathcal{G}$: $U\rightarrow V$, $V\rightarrow U$, or
there exists a vertex $W$ such that $U \rightarrow W \leftarrow V$. \cite%
{lauritzen-book}, Proposition 3.25 establishes that 
\begin{equation}
\mathbf{U} \perp \!\!\!\perp_{\mathcal{G}} \mathbf{W} \mid \mathbf{Z}
\Leftrightarrow \mathbf{U} \perp_{\left\lbrace \mathcal{G}_{\an_{ \mathcal{G}%
}(\mathbf{U}\cup \mathbf{W}\cup\mathbf{Z})}\right\rbrace^{m}} \mathbf{W}
\mid \mathbf{Z}  \label{eq:moral_equiv}
\end{equation}

\medskip

\subsection{Causal graphical models}

A Bayesian Network $\mathcal{M}\left( \mathcal{G}\right) $ represented by a
DAG $\mathcal{G}$ is a statistical model that identifies the vertex set $%
\mathbf{V}$ with a random vector and assumes that the law $P$ of $\mathbf{V}$
satisfies the Local Markov Property: $V\perp \!\!\!\perp \nd_{\mathcal{G}%
}\left( V\right) \text{ }|\text{ }\pa_{\mathcal{G}}\left( V\right) \text{
under }P$ for all $V\in \mathbf{V}$, where $A\perp \!\!\!\perp B|C$ stands
for conditional independence of $A$ and $B$ given $C.$ Assuming, as we will
throughout, that $P$ admits a density $f$ with respect to some dominating
measure, the Local Markov Property implies that 
\begin{equation}
f\left( \mathbf{v}\right) =\prod\limits_{V_{j}\in \mathbf{V}}f\left\{
v_{j}\mid \pa_{\mathcal{G}}(v_{j})\right\} ,  \label{eq-factorization}
\end{equation}%
where $\pa_{\mathcal{G}}(v_{j})$ is the value taken by $\pa_{\mathcal{G}%
}\left( V_{j}\right)$ when $\mathbf{V}=\mathbf{v}.$

Throughout the paper we will assume a causal agnostic graphical model %
\citep{spirtes, mediation} represented by a DAG $\mathcal{G}$. The model
identifies the vertex set of $\mathcal{G}$ with a factual random vector $%
\mathbf{V\equiv }(V_{1},\dots ,V_{s})$ and assumes that: (i) the law $P$ of $%
\mathbf{V}$ follows model $\mathcal{M}\left( \mathcal{G}\right) $ and (ii)
for any $A\in \mathbf{V}$, $\mathbf{L}\subset \nd_{\mathcal{G}}(A)$ and $%
\pi(A\mid \mathbf{L})$ a conditional law for $A$ given $\mathbf{L}$, the
intervention density $f_{\pi}\left( \mathbf{v}\right) $ of the variables in
the graph when, possibly contrary to fact, the value of $A$ is drawn from
the law $\pi(A\mid \mathbf{L})$ is given by 
\begin{equation}
f_{\pi}\left( \mathbf{v}\right) =\pi(a\mid \mathbf{l})\prod\limits_{V_{j}\in 
\mathbf{V}\setminus \{A\}}f\left\{ v_{j}\mid \pa_{\mathcal{G}%
}(v_{j})\right\} ,  \label{eq:g-form}
\end{equation}%
where $a$ and $\mathbf{l}$ are the values taken by $A$ and $\mathbf{L}$ when 
$\mathbf{V}$ is equal to $\mathbf{v}$. Formula \eqref{eq:g-form} is known as
the g-formula \citep{robins1986}, the manipulated density formula %
\citep{scheines} or the truncated factorization formula %
\citep{pearl-causality}. The conditional law $\pi$ designates the, possibly
random, policy or dynamic treatment regime. A non-random regime that assigns
the value $d\left( \mathbf{L}\right) $ to $A$ corresponds to the point mass
conditional law $\pi(a\mid \mathbf{l})=I_{d(\mathbf{l})}(a).$ In particular,
a constant function $d(\mathbf{L})=a$ corresponds to a non-random static
regime that sets $A$ to $a$.

Associating a given vertex $Y\in \de_{\mathcal{G}}\left( A\right) $ with the
outcome or reward of interest, the value of the policy $\pi $, denoted
throughout as $\chi _{\pi }(P;\mathcal{G}),$ is defined as the mean of the
outcome under the intervention law $f_{\pi }.$ By the factorizations $\left( %
\ref{eq-factorization}\right) $ and $\left( \ref{eq:g-form}\right) ,$ the
Radom-Nykodim theorem gives 
\begin{equation*}
\chi _{\pi }(P;\mathcal{G})=E_{P}\left\{ \frac{\pi \left( A\mid \mathbf{L}%
\right) }{f\left\{ A\mid \pa_{\mathcal{G}}(A)\right\} }Y\right\} .
\end{equation*}%
Furthermore, the Local Markov property implies that 
\begin{eqnarray*}
\chi _{\pi }(P;\mathcal{G}) &=&\int y\pi (a\mid \mathbf{l})f(y\mid a,\mathbf{%
l},\mathbf{\pa})f(\mathbf{l},\pa)d(y,a,\pa,\mathbf{l}\setminus \pa) \\
&=&E_{P}\left( E_{\pi ^{\ast }}\left[ E_{P}\left\{ Y\mid A,\pa_{\mathcal{G}%
}(A),\mathbf{L}\right\} \mid \pa_{\mathcal{G}}(A),\mathbf{L}\right] \right)
\end{eqnarray*}%
where $E_{P}\left( \cdot |\cdot \right) $ stands for conditional mean under $%
P$ and $E_{\pi ^{\ast }}\left( \cdot |\cdot \right) $ stands for conditional
mean under the conditional law of $A$ given $\mathbf{L}$ and $\pa_{\mathcal{G%
}}(A)$ defined as $\pi ^{\ast }\{A\mid \pa_{\mathcal{G}}(A),\mathbf{L}%
\}\equiv \pi (A\mid \mathbf{L})$.

In this article we study inference about $\chi _{\pi}(P;\mathcal{G})$ when
only a subset $\mathbf{N}$ of $\mathbf{V}$ is observable. The inferential
problem is thus defined by the following three assumptions: (i) the law $P$
follows the Bayesian Network $\mathcal{M}\left( \mathcal{G}\right) ,$ (ii)
only the sub-vector $\mathbf{N}$ of $\mathbf{V}$ is observable on a random
sample from the marginal law of $\mathbf{N}$ under $P$ and (iii) the
parameter of interest is the functional $\chi _{\pi}(P;\mathcal{G}).$ We
have motivated this inferential task with the causal agnostic model but we
could as well have motivated it with any other existing causal graphical
model in which the law $P$ of $\mathbf{V}$ is restricted only by the Local
Markov Property and the parameter representing the value under the
intervention $\pi$ coincides with the functional $\chi _{\pi}(P;\mathcal{G}).
$ Two such models are the non-parametric structural equation model with
independent errors \citep{pearl-causality} and the finest fully randomized
causally interpreted tree structured graph model \citep{robins1986}.

Throughout the paper we will assume that both $A$ and $Y$ are observable,
that $A$ is an ancestor of $Y$, and that we are interested in estimating the
value of policies $\pi (A\mid \mathbf{L})$ that depend on a, possibly empty,
vector $\mathbf{L}$ of observable non-descendants of $A$. That is, we assume
(i) $A\in \an_{\mathcal{G}}(Y)$, (ii) $\left\{ A,Y\right\} \cup \mathbf{L}%
\subset \mathbf{N}$ and (iiii) $\mathbf{L}\subset \nd_{\mathcal{G}}\left(
A\right) $. For ease of reference we refer to (i), (ii) and (iii) as the
inclusion assumptions. In all the definitions and results that follow in the
rest of the paper we will assume that that $\mathcal{G}$ is a DAG with
vertex set $\mathbf{V}$, and that $\left( A,Y,\mathbf{L},\mathbf{N}\right) $
satisfy the inclusion assumptions. Moreover, to avoid distracting measure
theoretic complications, we will assume throughout that $A$ takes values in
a finite set $\mathcal{A}$.

\section{Adjustment sets}

\label{sec:adj} \cite{shpitser-adjustment} and \cite{maathuis2015} gave the
following definition of adjustment set for static regime. We add the
appellative static to the name adjustment set to distinguish this set from
adjustment sets for, possibly random, dynamic regimes that we will define
subsequently.

\begin{definition}
A set $\mathbf{Z\subset V\backslash }\left\{ A,Y\right\} $ is a static
adjustment set relative to $A,Y$ in $\mathcal{G}$ if for all fixed $a$,
under all $P\in \mathcal{M}\left( \mathcal{G}\right) $ 
\begin{equation}
E_{P}\left[ E_{P}\left\{ I_{(-\infty ,y]}(Y)|A=a,\pa_{\mathcal{G}%
}(A)\right\} \right] =E_{P}\left[ E_{P}\left\{ I_{(-\infty ,y]}(Y)|A=a,%
\mathbf{Z}\right\} \right] \quad \text{for all }y\in \mathbb{R}.  \notag
\end{equation}
\end{definition}

The definition implies that for a static adjustment set $\mathbf{Z}$, the
value function $\chi _{\pi _{a}}(P;\mathcal{G})$ of the non-random static
regime $\pi _{a}\left( A\right) \equiv I_{a}\left( A\right) $ that sets $a$
to $A$ admits a representation as the iterated conditional expectation $%
E_{P}\left\{ E_{P}\left( Y\mid A=a,\mathbf{Z}\right) \right\} $. The
back-door criterion \citep{pearl-causality} is a well known graphical
condition that is sufficient, but not necessary, for $\mathbf{Z}$ to be a
static adjustment set. \cite{shpitser-adjustment} gave a necessary and
sufficient graphical condition for $\mathbf{Z}$ to be a static adjustment
set. \cite{vanAI} provide an alternative, constructive, graphical
characterization of static adjustment sets.

We now extend the preceding definition to accommodate, possibly random, $%
\mathbf{L}-$dependent policies.

\begin{definition}
Let $\mathbf{L}\subset \nd_{\mathcal{G}}(A)$. A set $\mathbf{Z\subset
V\backslash }\left\{ A,Y\right\} $ is an $\mathbf{L}$ dynamic adjustment set
with respect to $A,Y$ in $\mathcal{G}$ if $\mathbf{L}\subset \mathbf{Z}$ and
for all conditional laws $\pi (A\mid \mathbf{L})$ for $A$ given $\mathbf{L}$%
, all $P\in \mathcal{M}\left( \mathcal{G}\right) $ and all $y\in \mathbb{R}$%
. 
\begin{equation}
E_{P}\left( E_{\pi ^{\ast }}\left[ E_{P}\left\{ I_{(-\infty ,y]}(Y)\mid A,\pa%
_{\mathcal{G}}(A),\mathbf{L}\right\} \mid \pa_{\mathcal{G}}(A),\mathbf{L}%
\right] \right) =E_{P}\left( E_{\pi _{\mathbf{Z}}^{\ast }}\left[
E_{P}\left\{ I_{(-\infty ,y]}(Y)\mid A,\mathbf{Z}\right\} \mid \mathbf{Z}%
\right] \right) ,  \label{eq:L-dep}
\end{equation}%
where $\pi _{\mathbf{Z}}^{\ast }(A\mid \mathbf{Z})\equiv \pi (A\mid \mathbf{L%
})$ and, recall, $\pi ^{\ast }\left\{ A\mid \pa_{\mathcal{G}}(A),\mathbf{L}%
\right\} \equiv \pi (A\mid \mathbf{L})$.
\end{definition}

Suppose that at the stage of planning a study aimed at estimating different $%
\mathbf{L}-$ dependent policies, and having postulated a causal graphical
model, the investigator acknowledges that due to practical, ethical or cost
reasons, she can only hope to observe a subset $\mathbf{N}$ of the variables
in $\mathcal{G}$. Furthermore, suppose that $\mathbf{N}$ includes at least
one $\mathbf{L}$ dynamic adjustment set $\mathbf{Z}$. She may then choose to
measure, in addition to $A$ and $Y,$ solely the variables $\mathbf{Z}$, as
these variables suffice to identify the policy value $\chi _{\pi }(P;%
\mathcal{G})$ with the functional 
\begin{equation*}
\chi _{\pi ,\mathbf{Z}}(P;\mathcal{G})\equiv E_{P}\left[ E_{\pi _{\mathbf{Z}%
}^{\ast }}\left\{ E_{P}\left( Y\mid A,\mathbf{Z}\right) \mid \mathbf{Z}%
\right\} \right] ,
\end{equation*}%
and subsequently proceed to estimate $\chi _{\pi ,\mathbf{Z}}(P;\mathcal{G})$
non-parametrically as further explained in Section \ref{sec:nonparam}. Note
that this strategy effectively uses the causal graphical model solely as an
aid to identify adjustment sets at the design stage, but for robustness
against model misspecification, it avoids exploiting the restrictions
implied by the Bayesian Network $\mathcal{M}\left( \mathcal{G}\right) $ to
either identify $\chi _{\pi }(P;\mathcal{G})$ with a formula different from
adjustment formula $\left( \ref{eq:L-dep}\right) $ or to improve efficiency
in the estimation of the functionals $\chi _{\pi ,\mathbf{Z}}(P;\mathcal{G}%
). $

The preceding formulation raises the following questions: (1) given a graph $%
\mathcal{G}$ and a subset $\mathbf{N}$ of its vertices, how can we tell if
an $\mathbf{L}$ dynamic adjustment set that is a subset of $\mathbf{N}$
exists? and (2) if several different observable $\mathbf{L}$ dynamic
adjustment sets exist, which one should one measure? Our goal is to answer
these question assuming that the basis for comparing adjustment sets $%
\mathbf{Z}$ is the variance of the limiting distribution of the
non-parametric estimators of $\chi _{\pi ,\mathbf{Z}}(P;\mathcal{G})$. To
formally study these problems, we start with the following definitions.

\begin{definition}
The pair $\left( \mathbf{L},\mathbf{N}\right) $ is said to be an admissible
pair with respect to $A,Y$ in $\mathcal{G}$ if there exists an adjustment
set $\mathbf{Z}$ with respect to $A,Y$ in $\mathcal{G}$ such that $\mathbf{L}%
\subset \mathbf{Z}\subset \mathbf{N}.$
\end{definition}

\begin{definition}
An $\mathbf{L}$ dynamic adjustment set $\mathbf{Z}$ with respect to $A,Y$ in 
$\mathcal{G}$ that is a subset of $\mathbf{N}$ is said to be an $\mathbf{L}-%
\mathbf{N}$ dynamic adjustment set. An $\mathbf{L}-\mathbf{N}$ dynamic
adjustment set $\mathbf{Z}$ is said to be minimal if no strict subset of $%
\mathbf{Z}$ is an $\mathbf{L}-\mathbf{N}$ dynamic adjustment set. An $%
\mathbf{L}-\mathbf{N}$ dynamic adjustment set $\mathbf{Z}$ is said to be
minimum if there exists no $\mathbf{L}-\mathbf{N}$ dynamic adjustment set
with cardinality strictly smaller than the cardinality of $\mathbf{Z}$.
\end{definition}

\begin{definition}
An $\mathbf{L}-\mathbf{N}$ static adjustment set $\mathbf{Z}$ with respect
to $A,Y$ in $\mathcal{G}$ is a static adjustment set with respect to $A,Y$
in $\mathcal{G}$ that satisfies $\mathbf{L}\subset \mathbf{Z}\subset \mathbf{%
N.}$ An $\mathbf{L}-\mathbf{N}$ static adjustment set $\mathbf{Z}$ is said
to be minimal if no strict subset of $\mathbf{Z}$ is an $\mathbf{L}-\mathbf{N%
}$ static adjustment set. An $\mathbf{L}-\mathbf{N}$ static adjustment set $%
\mathbf{Z}$ is said to be minimum if there exists no $\mathbf{L}-\mathbf{N}$
static adjustment set of cardinality strictly smaller than the cardinality
of $\ \mathbf{Z}$.
\end{definition}

The following proposition establishes that the class of $\mathbf{L}-\mathbf{N%
}$ static (minimal, minimum) adjustment sets and the class of $\mathbf{L}-%
\mathbf{N}$ (minimal, minimum) dynamic adjustment sets coincide.
Additionally, it establishes that minimal $\mathbf{L}-\mathbf{N}$ dynamic
adjustment sets are subsets of the set of ancestors of $\{A,Y\}\cup \mathbf{%
L.}$

\begin{proposition}
\label{prop:equiv_dtr} $\ $

\begin{enumerate}
\item $\mathbf{Z}$ is an $\mathbf{L}-\mathbf{N}$ dynamic adjustment set with
respect to $A,Y$ in $\mathcal{G}$ if and only if $\ \mathbf{Z}$ is an $%
\mathbf{L}-\mathbf{N}$ static adjustment set with respect to $A,Y$ in $%
\mathcal{G}$.

\item If $\mathbf{Z}$ is a minimal $\mathbf{L}-\mathbf{N}$ dynamic
adjustment set with respect to $A,Y$ in $\mathcal{G}$ then $\mathbf{Z}%
\subset \an_{\mathcal{G}}(\{A,Y\}\cup \mathbf{L})$.

\item $\mathbf{Z}$ is a minimal $\mathbf{L}-\mathbf{N}$ dynamic adjustment
set with respect to $A,Y$ in $\mathcal{G}$ if and only if $\mathbf{Z}$ is a
minimal $\mathbf{L}-\mathbf{N}$ static adjustment set with respect to $A,Y$
in $\mathcal{G}$.

\item $\mathbf{Z}$ is a minimum $\mathbf{L}-\mathbf{N}$ dynamic adjustment
set with respect to $A,Y$ in $\mathcal{G}$ if and only if $\mathbf{Z}$ is a
minimum $\mathbf{L}-\mathbf{N}$ static adjustment set with respect to $A,Y$
in $\mathcal{G}$.
\end{enumerate}
\end{proposition}

Part 1) of Proposition \ref{prop:equiv_dtr} strengthens Theorem 2 of \cite%
{kuroki-miyawa} which establishes that if $\mathbf{Z}$ satisfies the
back-door criterion \citep{pearl-causality} and $\mathbf{L}\subset \mathbf{Z}%
\subset \mathbf{N} $ then $\mathbf{Z}$ is an $\mathbf{L}-\mathbf{N}$
non-random dynamic adjustment set with respect to $A,Y$ in $\mathcal{G}$ if
by such adjustment sets we mean those that satisfy $\left( \ref{eq:L-dep}%
\right) $ for any point mass probability $\pi(A\mid \mathbf{L})$ $=I_{d(%
\mathbf{L})}(A)$ and any given $d.$

Combining Part 1) of Proposition \ref{prop:equiv_dtr} and a result by \cite%
{vanAI} we can also give an answer to the first question raised above.
Specifically, there exists an $\mathbf{L}-\mathbf{N}$ dynamic adjustment set
if and only if there exists an $\mathbf{L}-\mathbf{N}$ static adjustment set
with respect to $A,Y$ in $\mathcal{G}$, i.e. if and only if the pair $\left( 
\mathbf{L},\mathbf{N}\right) $ is admissible. \cite{vanAI} gave the
following necessary and sufficient graphical condition for the pair $\left( 
\mathbf{L},\mathbf{N}\right) $ to be admissible: the set $\left[ \an_{%
\mathcal{G}}(\{A,Y\}\cup \mathbf{L})\cap \mathbf{N}\right] \setminus \forb%
(A,Y,\mathcal{G})$ is an $\mathbf{L}-\mathbf{N}$ static adjustment set,
where the so-called forbidden set is defined as $\forb(A,Y,\mathcal{G}%
)\equiv \de_{\mathcal{G}}\left( \cn(A,Y,\mathcal{G})\right) \cup \left\{
A\right\} $ with $\cn(A,Y,\mathcal{G})$ defined as the set of all vertices
that lie on a causal path between a vertex in $A$ and $Y$ and are not equal
to $A$. \cite{vanAI} additionally provided a polynomial time algorithm to
test this condition.

In addition to providing a graphical test of $\left( \mathbf{L},\mathbf{N}%
\right) $ admissibility, \cite{vanAI} provided a constructive graphical
characterization of $\mathbf{L}-\mathbf{N}$ static adjustment sets when
these exist. Moreover, they provided a polynomial time algorithm to find one 
$\mathbf{L}-\mathbf{N}$ static adjustment set, one minimal $\mathbf{L}-%
\mathbf{N}$ static adjustment sets and one minimum $\mathbf{L}-\mathbf{N}$
static adjustment set and an algorithm with polynomial time latency to list
all minimal $\mathbf{L}-\mathbf{N}$ static adjustment sets and all minimum $%
\mathbf{L}-\mathbf{N}$ static adjustment sets. Proposition \ref%
{prop:equiv_dtr} implies that the results of \cite{vanAI} are equally
applicable to find $\mathbf{L}-\mathbf{N}$ dynamic adjustment sets.

\section{Non-parametric estimation of a policy value}

\label{sec:nonparam}

We begin this section highlighting the elements of the asymptotic theory for
non-parametric estimators of $\chi _{\pi ,\mathbf{Z}}(P;\mathcal{G})$ that
are relevant to our derivations. An estimator $\widehat{\gamma }$ of a
parameter $\gamma \left( P\right) $ based on $n$ independent identically
distributed random variables $\mathbf{V}_{1},\dots ,\mathbf{V}_{n}$ of $%
\mathbf{V}$ is said to be asymptotically linear at $P$ if there exists a
random variable $\varphi _{P}\left( \mathbf{V}\right) $, called the
influence function of $\gamma (P)$ at $P$, which when $\mathbf{V} \sim $ $P$%
, has mean zero and finite variance and is such that $n^{1/2}\left\{ 
\widehat{\gamma }-\gamma \left( P\right) \right\}
=n^{-1/2}\sum_{i=1}^{n}\varphi _{P}\left( \mathbf{V}_{i}\right) +o_{p}(1).$
By the Central Limit Theorem, for any asymptotically linear estimator $%
\widehat{\gamma },$ it holds that $n^{1/2}\left\{ \widehat{\gamma }-\gamma
\left( P\right) \right\} $ converges in distribution to a mean zero Normal
distribution with variance $var_{P}\left\{ \varphi _{P}\left( \mathbf{V}%
_{i}\right) \right\} $. Given a collection of probability laws $\mathcal{P}$
for $\mathbf{V}$, an estimator of $\widehat{\gamma }$ of $\gamma \left(
P\right) $ is said to be regular at one $P$ if its convergence to $\gamma
\left( P\right) $ is locally uniform at $P$ in $\mathcal{P}$ %
\citep{van1998asymptotic}.

It is well known that estimators of $\chi _{\pi ,\mathbf{Z}}(P;\mathcal{G})$
that are regular and asymptotically linear at all $P$ in any given model $%
\mathcal{P}$ that makes at most `complexity' type assumptions on 
\begin{equation*}
b\left( A,\mathbf{Z};P\right) \equiv E_{P}\left( Y|A,\mathbf{Z}\right) \quad 
\text{and/or}\quad f\left( A\mid \mathbf{Z}\right)
\end{equation*}%
have the influence function $\psi _{P,\pi }\left( \mathbf{Z};\mathcal{G}%
\right) $ given by 
\begin{equation}
\psi _{P,\pi }\left( \mathbf{Z};\mathcal{G}\right) \equiv \frac{\pi (A\mid 
\mathbf{L})}{f\left( A\mid \mathbf{Z}\right) }\left\{ Y-b\left( A,\mathbf{Z}%
;P\right) \right\} +E_{\pi _{\mathbf{Z}}^{\ast }}\left\{ b\left( A,\mathbf{Z}%
;P\right) \mid \mathbf{Z}\right\} -\chi _{\pi }\left( P,\mathcal{G}\right) .
\label{eq:inf_fc_sing}
\end{equation}%
Here and throughout, for conciseness, we avoid writing the arguments $A$ and 
$Y$ in the function $\psi _{P,\pi }\left( \cdot ,P\right) .$ Examples of
complexity type assumptions are the assumptions that the functions $b\left(
A,\mathbf{Z};P\right) $ and/or $f\left( A\mid \mathbf{Z}\right) $ belong to
some smooth function class such as a Holder ball, or to a function class
with a Rademacher complexity that satisfies certain bounds. Examples of
estimation strategies that make complexity type assumptions include: the
inverse probability weighted estimator $\widehat{\chi }_{\pi ,IPW}=\mathbb{P}%
_{n}\left\{ \widehat{f}\left( A\mid \mathbf{Z}\right) ^{-1}\pi (A\mid 
\mathbf{L})Y\right\} ,$ where $\widehat{f}\left( A|Z\right) $ is a
non-parametric smoothing type, e.g. series or kernel based, estimator of $%
f\left( A\mid \mathbf{Z}\right) $ \citep{hirano}, the outcome regression
estimator $\mathbb{P}_{n}\left[ E_{\pi _{\mathbf{Z}}^{\ast }}\left\{ 
\widehat{b}\left( A,\mathbf{Z}\right) \mid \mathbf{Z}\right\} \right] $
where $\widehat{b}$ is a non-parametric smoothing type estimator of $b$ %
\citep{hahn} and the doubly-robust estimator 
\citep{vanderlaan, dudik,
chernozhukov2018double, smucler} with both $f\left( A\mid \mathbf{Z}\right) $
and $b\left( A,\mathbf{Z}\right) $ estimated via smoothing techniques. Note
that models that only place complexity type assumptions on $f\left( A\mid 
\mathbf{Z}\right) $ and/or $b\left( A,\mathbf{Z}\right) ,$ ignore any
restriction that could be possibly implied on the law of $\left( A,\mathbf{Z}%
,Y\right) $ by the Bayesian Network $\mathcal{M}\left( \mathcal{G}\right) .$
We will refer to estimators that are regular and asymptotically linear with
influence function $\psi _{P,\pi }\left( \mathbf{Z};\mathcal{G}\right) $\
defined in $\left( \ref{eq:inf_fc_sing}\right) $ as non-parametric
estimators and, for brevity, we will designate them as NP-$\mathbf{Z}$
estimators. It follows from the discussion above that all NP-$\mathbf{Z}$
estimators $\widehat{\chi }_{\pi ,\mathbf{Z}}$ satisfy that $\sqrt{n}\left\{ 
\widehat{\chi }_{\pi ,\mathbf{Z}}-\chi _{\pi }\left( P;\mathcal{G}\right)
\right\} $ converges in distribution to $N\left\{ 0,\sigma _{\pi ,\mathbf{Z}%
}^{2}\left( P\right) \right\} $ where $\sigma _{\pi ,\mathbf{Z}}^{2}\left(
P\right) \equiv var_{P}\left\{ \psi _{P,\pi }\left( \mathbf{Z};\mathcal{G}%
\right) \right\} .$

On the class of static adjustment sets we define the preorder (a reflexive
and transitive binary relation) $\preceq $ as follows 
\begin{equation*}
\mathbf{Z}^{\prime }\preceq \mathbf{Z}\text{ if and only if }\sigma _{\pi ,%
\mathbf{Z}^{\prime}}^{2}\left( P\right) \leq \sigma _{\pi ,\mathbf{Z}}^{2}\left( P\right) \text{ for all }P\in \mathcal{M(G)}\text{ and all }\pi
(A\mid \mathbf{L})=I_{a}\left( A\right) \text{ for some }a
\end{equation*}
and on the class of $\mathbf{L}-\mathbf{N}$ dynamic adjustment sets we
define the preorder $\preceq_{\mathbf{L}}$ as follows%
\begin{equation*}
\mathbf{Z}^{\prime }\preceq_{\mathbf{L}}\mathbf{Z}\text{ if and only if }%
\sigma _{\pi ,\mathbf{Z}^{\prime}}^{2}\left( P\right) \leq \sigma _{\pi ,\mathbf{Z}}^{2}\left( P\right) \text{ for all }P\in \mathcal{M(G)}\text{ and
all }\pi (A\mid \mathbf{L})
\end{equation*}

\cite{eff-adj} showed that $\preceq $ is not a total preorder because there
exist graphs $\mathcal{G}$ and vertices $A,Y$ such that for some static
adjustment sets relative to $A$ and $Y$ in $\mathcal{G}$, say $\mathbf{Z}%
^{\prime }$ and $\mathbf{Z}$, it holds that $\sigma _{\pi ,\mathbf{Z}%
}^{2}\left( P\right) \leq \sigma _{\pi ,\mathbf{Z}^{\prime }}^{2}\left(
P\right) $ for some $P\in \mathcal{M(G)}$ but $\sigma _{\pi ,\mathbf{Z}%
}^{2}\left( P^{\prime }\right) >\sigma _{\pi ,\mathbf{Z}^{\prime
}}^{2}\left( P^{\prime }\right) $ for some other $P^{\prime }\in \mathcal{%
M(G)}$. A similar negative result was derived earlier by \cite{perkovic} for
comparing the variances of ordinary least squares estimators of the
coefficient of $A$ in the regression of $Y$ with covariates $A$ and the
adjustment set in question under the assumption that the causal graphical
model is linear, that is, that $\mathbf{V}=(V_{1},\dots, V_{s})$ satisfies $%
V_{i}= \sum_{V_{j}\in \pa_{\mathcal{G}}(V_{i})} \alpha_{ij}V_{j} +
\varepsilon_{i}, $ for $i\in \lbrace 1,\dots,s\rbrace$, where $\alpha_{ij}\in%
\mathbb{R}$ and $\varepsilon_{1},\dots,\varepsilon_{s}$ are jointly
independent random variables with zero mean and finite variance. However, 
\cite{perkovic} gave two graphical criteria for ordering certain pairs of
static adjustment sets in the aforementioned linear setting. \cite{eff-adj}
proved that the same graphical criterion applies in the non-parametric
setting. The preorder $\preceq_{\mathbf{L}}$ is not a total preorder because 
$\mathbf{Z}^{\prime }\preceq_{\mathbf{L}} \mathbf{Z}$ for some $\mathbf{L}$
implies $\mathbf{Z}^{\prime }\preceq \mathbf{Z}$. Our first result,
formalized in Lemmas \ref{lemma:supplementation} and \ref{lemma:deletion},
and in Proposition \ref{prop:compare_adj} in the next section extends the
graphical criteria of \cite{perkovic} and \cite{eff-adj} to $\mathbf{L}-%
\mathbf{N}$ dynamic adjustment sets.

\cite{perkovic} and \cite{eff-adj} in the linear and non-parametric settings
respectively showed that a globally optimal adjustment set $\mathbf{O}$
satisfying $\mathbf{O}\preceq \mathbf{Z}$ for any other static adjustment
set $\mathbf{Z}$ always exists and they provided a graphical
characterization of it. They also showed, by exhibiting counterexamples,
that in graphs with hidden variables which admit observable static
adjustment sets there may not exist an optimal static adjustment set among
the observable ones. Because static interventions are a special of dynamic
interventions with $\mathbf{L}=\emptyset $, the same assertions hold for $%
\mathbf{L-N}$ dynamic adjustment sets. In Section \ref{sec:new_graph} we
will provide a sufficient graphical condition for a globally optimal $%
\mathbf{L}-\mathbf{N}$ dynamic adjustment set to exist in graphs with hidden
variables. Subsequently, we will demonstrate that for $\left( \mathbf{L},%
\mathbf{N}\right) $ admissible pairs, an optimal adjustment set always
exists among minimal $\mathbf{L}-\mathbf{N}$ dynamic adjustment sets and we
will provide a graphical characterization of it. Our results extend the
results of \cite{eff-adj} who gave a graphical characterization of the
minimal optimal static adjustment set in graphs without hidden variables.
Finally, we will show that, for $\left( \mathbf{L},\mathbf{N}\right) $
admissible pairs, in the class of minimum $\mathbf{L-N}$ dynamic adjustment
sets there always exists an optimal one and we will provide a graphical
characterization of it. Our results can be applied, in particular, to
determine the optimal static adjustment set among the minimum static
adjustment sets. To our knowledge neither the proof of the existence of an
optimal minimum static or $\mathbf{L-N}$ dynamic adjustment set nor a
graphical characterization of it are available in the existing literature.

\section{Comparing dynamic adjustment sets}

\label{sec:compare} We start this section by establishing two Lemmas which
entail the possibility of ordering certain pairs of $\mathbf{L}-\mathbf{N}$
dynamic adjustment sets as indicated in the preceding section. These Lemmas
extend Lemmas 1 and 2 of \cite{eff-adj} from static to $\mathbf{L}-\mathbf{N}
$ dynamic adjustment sets. Throughout this section all results assume that $(%
\mathbf{L,N})$ is an admissible pair with respect to $A,Y$ in $\mathcal{G}$.
Also, all adjustment sets, static or dynamic are with respect to $A,Y$ in $%
\mathcal{G}$.

\begin{lemma}[Supplementation with precision variables]
\label{lemma:supplementation} Let $\mathbf{B}$ be an $\mathbf{L}-\mathbf{N}$
dynamic adjustment set and let $\mathbf{G}\subset \mathbf{V}$ satisfy 
\begin{equation*}
A\perp \!\!\!\perp _{\mathcal{G}}\mathbf{G}\mid \mathbf{B}.
\end{equation*}%
Then $\mathbf{G}\cup \mathbf{B}$ is also an $\mathbf{L}-\mathbf{N}$ dynamic
adjustment set and for any $\pi (A\mid \mathbf{L})$ and all $P\in \mathcal{M}%
\left( \mathcal{G}\right) $ 
\begin{equation}
\sigma _{\pi ,\mathbf{B}}^{2}\left( P\right) -\sigma _{\pi ,\mathbf{G},%
\mathbf{B}}^{2}\left( P\right) =\mathbf{1}^{\top }var_{P}(\mathbf{S})\mathbf{%
1}\geq 0,  \notag
\end{equation}%
where $\mathbf{S}=\left( S_{a}\right) _{a\in \mathcal{A}}$, 
\begin{equation*}
S_{a}\equiv \left\{ \frac{I_{a}(A)}{f(a\mid \mathbf{G},\mathbf{B})}%
-1\right\} \pi (a\mid \mathbf{L})\left\{ b(a,\mathbf{G},\mathbf{B};P)-b(a,%
\mathbf{B};P)\right\} ,
\end{equation*}%
and $\mathbf{1}$ is a vector of ones with dimension equal to $\#\mathcal{A}$%
. Moreover, 
\begin{align*}
& var_{P}(S_{a})=E_{P}\left\{ var_{P}\left[ b(a,\mathbf{G},\mathbf{B};P)\mid 
\mathbf{B}\right] \pi (a\mid \mathbf{L})^{2}\left[ \frac{1}{f(a\mid \mathbf{B%
})}-1\right] \right\} , \\
& cov_{P}(S_{a},S_{a^{\prime }})=-E_{P}\left[ \pi (a\mid \mathbf{L})\pi
(a^{\prime }\mid \mathbf{L})cov_{P}\left\{ b(a,\mathbf{G},\mathbf{B}%
;P),b(a^{\prime },\mathbf{G},\mathbf{B};P)|\mathbf{B}\right\} \right] .
\end{align*}
\end{lemma}

\begin{lemma}[Deletion of overadjustment variables]
\label{lemma:deletion} Let $\mathbf{G}\cup \mathbf{B}$ be an $\mathbf{L}-%
\mathbf{N}$ dynamic adjustment set with $\mathbf{G}$ and $\mathbf{B}$
disjoint and suppose 
\begin{equation*}
\mathbf{L}\subset \mathbf{G}\quad \text{and}\quad Y\text{ }\mathbf{\perp
\!\!\!\perp }_{\mathcal{G}}\text{ }\mathbf{B}\text{ }\mathbf{|}\text{ }%
\mathbf{G},A.
\end{equation*}%
Then $\mathbf{G}$ is also an $\mathbf{L}-\mathbf{N}$ dynamic adjustment set
and for any $\pi (A\mid \mathbf{L})$ and all $P\in \mathcal{M}\left( 
\mathcal{G}\right) $ 
\begin{align}
& \sigma _{\pi ,\mathbf{G},\mathbf{B}}^{2}\left( P\right) -\sigma _{\pi ,%
\mathbf{G}}^{2}\left( P\right) =  \notag \\
& \sum\limits_{a\in \mathcal{A}}\left( E_{P}\left[ \pi ^{2}(a\mid \mathbf{L}%
)f(a\mid \mathbf{G})var_{P}(Y\mid A=a,\mathbf{G})var_{P}\left\{ \frac{1}{%
f(a\mid \mathbf{G,B})}\mid A=a,\mathbf{G}\right\} \right] \right) \geq 0. 
\notag
\end{align}
\end{lemma}

A straightforward consequence of Lemmas \ref{lemma:supplementation} and \ref%
{lemma:deletion} is Proposition \ref{prop:compare_adj} below, which shows
that the graphical criteria to compare certain pairs of static adjustment
sets in \cite{perkovic} and \cite{eff-adj} is also valid for comparing $%
\mathbf{L}-\mathbf{N}$ dynamic adjustment sets.

\begin{proposition}
\label{prop:compare_adj} Suppose $\mathbf{G}$ and $\mathbf{B}$ are two $%
\mathbf{L}-\mathbf{N}$ dynamic adjustment sets such that%
\begin{equation}
A\perp \!\!\!\perp _{\mathcal{G}}\mathbf{G\backslash B}\mid \mathbf{B}\quad 
\text{and}\quad Y\perp \!\!\!\perp _{\mathcal{G}}\mathbf{B\backslash G}\mid 
\mathbf{G},A.  \notag
\end{equation}%
Then, for any $\pi (A\mid \mathbf{L})$ and all $P\in \mathcal{M}\left( 
\mathcal{G}\right) $ 
\begin{eqnarray*}
&&\sigma _{\pi ,\mathbf{B}}^{2}\left( P\right) -\sigma _{\pi ,\mathbf{G}%
}^{2}\left( P\right) =\mathbf{1}^{\top }var_{P}(\mathbf{S})\mathbf{1}+ \\
&&\sum\limits_{a\in \mathcal{A}}\left( E_{P}\left[ \pi ^{2}(a\mid \mathbf{L}%
)f(a\mid \mathbf{G})var_{P}(Y\mid A=a,\mathbf{G})var_{P}\left\{ \frac{1}{%
f(a\mid \mathbf{G,B})}\mid A=a,\mathbf{G}\right\} \right] \right) \geq 0,
\end{eqnarray*}%
where $\mathbf{S}$ is defined as in Lemma \ref{lemma:supplementation}.
\end{proposition}

\medskip 
\begin{proof}
Write $\sigma _{\pi,\mathbf{B}}^{2}-\sigma _{\pi,\mathbf{G}}^{2}=\sigma _{\pi,\mathbf{B}%
}^{2}-\sigma _{\pi,\mathbf{B\cup }\left( \mathbf{G\backslash B}\right)
}^{2}+\sigma _{\pi,\mathbf{G\cup }\left( \mathbf{B\backslash G}\right)
}^{2}-\sigma _{\pi,\mathbf{G}}^{2}$ and apply Lemmas \ref{lemma:supplementation}
and \ref{lemma:deletion}.
\end{proof}

Define 
\begin{equation*}
\mathbf{O}(A,Y,\mathcal{G})\equiv \pa_{\mathcal{G}}(\cn(A,Y,\mathcal{G}%
))\setminus \forb(A,Y,\mathcal{G})\quad \text{and}\quad \mathbf{O}(A,Y,%
\mathbf{L},\mathcal{G})\equiv \mathbf{O}(A,Y,\mathcal{G})\cup \mathbf{L}.
\end{equation*}
\cite{perkovic} and \cite{eff-adj} in the linear setting and the
non-parametric setting respectively showed that $\mathbf{O}(A,Y,\mathcal{G})$
is the globally optimal static adjustment set in graphs with no hidden
variables$\mathbf{.}$ We will now establish that the set $\mathbf{O}(A,Y,%
\mathbf{L},\mathcal{G})$ is a globally optimal $\mathbf{L}-\mathbf{N}$
dynamic adjustment set in graphs with no hidden variables, i.e. when $%
\mathbf{N}=\mathbf{V.}$

\begin{proposition}
\label{prop:opt_adj}Suppose that $\mathbf{N}=\mathbf{V}$ where $\mathbf{V}$
is the set of all the vertices in $\mathcal{G}$. Then $\mathbf{O}\equiv 
\mathbf{O}(A,Y,\mathbf{L},\mathcal{G})$ is an $\mathbf{L}-\mathbf{N}$
dynamic adjustment set and for any other $\mathbf{L}-\mathbf{N}$ dynamic
adjustment set $\mathbf{Z}$ it holds that 
\begin{equation*}
A\perp \!\!\!\perp _{\mathcal{G}}\mathbf{O}\setminus \mathbf{Z}\mid \mathbf{Z%
}\quad \text{and}\quad Y\perp \!\!\!\perp _{\mathcal{G}}\mathbf{Z}\setminus 
\mathbf{O}\mid A,\mathbf{O}.
\end{equation*}%
Consequently, $\mathbf{O}(A,Y,\mathbf{L},\mathcal{G})\preceq_{\mathbf{L}}%
\mathbf{Z}$ for any $\mathbf{L}-\mathbf{N}$ dynamic adjustment set $\mathbf{Z%
}$.
\end{proposition}

In Section \ref{sec:new_graph} we provide an alternative graphical
characterization of $\mathbf{O}(A,Y,\mathbf{L},\mathcal{G})$ as the set of
neighbors of $Y$ in a suitably constructed undirected graph.

\section{Graphical characterizations}

\label{sec:new_graph}

Assuming that $(\mathbf{L,N})$ is an admissible pair with respect to $A,Y$
in $\mathcal{G}$, in this section we will define an undirected graph which
will be the basis for our graphical criteria for characterizing the globally
optimal (when it exists), the optimal minimal and the optimal minimum $%
\mathbf{L}-\mathbf{N}$ dynamic adjustment sets. The construction of this new
graph relies on a result in \cite{vanAI} which, for completeness, we state
next before the definition of the aforementioned undirected graph. In what
follows, following \cite{vanAI}, we define the proper back-door graph $%
\mathcal{G}^{pbd}(A,Y)$ as the DAG formed by removing from $\mathcal{G}$ the
first edge of every causal path from $A$ to $Y$.

\begin{theorem}[Theorem 1 from \protect\cite{vanAI}]
\label{theo:van} The set $\mathbf{Z}$ is an $\mathbf{L}-\mathbf{N}$ static
adjustment set with respect to $A,Y$ in $\mathcal{G}$ if and only if (1) $%
Y\perp \!\!\!\perp _{\mathcal{G}^{pbd}(A,Y)}A\mid \mathbf{Z}$, (2) $\mathbf{Z%
}\cap \forb(A,Y,\mathcal{G})=\emptyset $, and (3) $\mathbf{L}\subset \mathbf{%
Z}\subset \mathbf{N}$.
\end{theorem}

\begin{definition}
\label{def:H} Let 
\begin{equation*}
\mathcal{H}^{0}(A,Y,\mathbf{L},\mathcal{G})\equiv \left\{ \mathcal{G}_{\an_{%
\mathcal{G}}(\{A,Y\}\cup \mathbf{L})}^{pbd}(A,Y)\right\} ^{m}
\end{equation*}%
and 
\begin{equation*}
\ignore(A,Y,\mathbf{L},\mathbf{N},\mathcal{G})\equiv \left\{ \an_{\mathcal{G}%
}(\{A,Y\}\cup \mathbf{L})\setminus \{A,Y\}\right\} \cap \left\{ \mathbf{N}%
^{c}\cup \forb(A,Y,\mathcal{G})\right\} .
\end{equation*}%
The non-parametric adjustment efficiency graph associated with $A,Y,\mathbf{L%
},\mathbf{N}$ in $\mathcal{G}$ is defined as the undirected graph, denoted
with $\mathcal{H}^{1}(A,Y,\mathbf{L},\mathbf{N},\mathcal{G}),$ constructed
from $\mathcal{H}^{0}(A,Y,\mathbf{L},\mathcal{G})$ by (1) removing all
vertices in $\ignore(A,Y,\mathbf{L},\mathbf{N},\mathcal{G})$, (2) adding an
edge between any pair of remaining vertices if they were connected in $%
\mathcal{H}^{0}(A,Y,\mathbf{L},\mathcal{G})$ by a path with vertices in $%
\ignore(A,Y,\mathbf{L},\mathbf{N},\mathcal{G})$ and (3) adding an edge: (i)
between $A$ and each vertex in $\mathbf{L}$ and, (ii) between $Y$ and each
vertex in $\mathbf{L}$.
\end{definition}

For conciseness, unless unclear, throughout we will drop $\mathbf{L}$ and $%
\mathbf{N}$ from $\ignore(A,Y,\mathbf{L},\mathbf{N},\mathcal{G})$ and we
will also write $\mathcal{H}^{0}$ and $\mathcal{H}^{1}$ instead of $\mathcal{%
H}^{0}(A,Y,\mathbf{L},\mathcal{G})$ and $\mathcal{H}^{1}(A,Y,\mathbf{L},%
\mathbf{N},\mathcal{G})$.

\cite{textor12} used the undirected graph $\mathcal{H}^{0}$ as the basis for
a graphical characterization of minimal static adjustment sets when $\mathbf{%
N}=\mathbf{V.}$ We will show later that our construction of $\mathcal{H}^{1}$
entails, among other characterizations, a graphical criterion that extends
the one in \cite{textor12} to minimal $\mathbf{L}-\mathbf{N}$ dynamic
adjustment sets and sets $\mathbf{N}$ that can be a strict subset of $%
\mathbf{V.}$ In addition, it entails the graphical characterization of a
globally optimal $\mathbf{L}-\mathbf{N}$ dynamic adjustment set for $\left( 
\mathbf{L},\mathbf{N}\right) $ admissible pairs when $\mathbf{N}\subset \an_{%
\mathcal{G}}(\{A,Y\}\cup \mathbf{L}).$ The heuristics behind the
construction of $\mathcal{H}^{1}$ are as follows. If $\mathbf{N}\subset \an_{%
\mathcal{G}}(\{A,Y\}\cup \mathbf{L})$ then any $\mathbf{L}-\mathbf{N}$
dynamic adjustment set must be a subset of $\an_{\mathcal{G}}(\{A,Y\}\cup 
\mathbf{L}).$ On the other hand, even if $\mathbf{N}\not \subset \an_{%
\mathcal{G}}(\{A,Y\}\cup \mathbf{L})$, Proposition \ref{prop:equiv_dtr}
established that all minimal, and consequently all minimum, $\mathbf{L}-%
\mathbf{N}$ dynamic adjustment sets are subsets of $\an_{\mathcal{G}%
}(\{A,Y\}\cup \mathbf{L})$. Now, suppose that $\mathbf{C}$ satisfies $%
\mathbf{L}\subset \mathbf{C}\subset \mathbf{N}$, $\mathbf{C}\cap \forb(A,Y,%
\mathcal{G})=\emptyset $ and $\mathbf{C}$ is an $A-Y$ cut in the moralized
graph $\mathcal{H}^{0}$ of the proper back-door graph $\mathcal{G}_{\an_{%
\mathcal{G}}(\{A,Y\}\cup \mathbf{L})}^{pbd}(A,Y).$ By Theorem \ref{theo:van}%
, the moralization property \eqref{eq:moral_equiv} and Proposition \ref%
{prop:equiv_dtr}, the cut $\mathbf{C}$ is an $\mathbf{L}-\mathbf{N}$ dynamic
adjustment set and is a subset of $\an_{\mathcal{G}}(\{A,Y\}\cup \mathbf{L})$%
. Next note that variables in $\ignore(A,Y,\mathcal{G})$ are either hidden
or forbidden and hence cannot be part of any $\mathbf{L}-\mathbf{N}$ dynamic
adjustment set. Steps 1 and 2 of Definition \ref{def:H} are similar in
spirit to a latent projection \citep{verma-pearl, admg} on $\mathbf{V}(%
\mathcal{H}^{0})\setminus \ignore(A,Y,\mathcal{G})$. The so called latent
projection operation in DAGs marginalizes DAGs over hidden variables while
preserving d-separation relations between the observable variables. Lemma %
\ref{lemma:indep_equiv} below establishes that $\mathcal{H}^{1}$ preserves
the separations in $\mathcal{H}^{0}$ between variables that lie in $\mathbf{V%
}(\mathcal{H}^{1})$, the vertex set of $\mathcal{H}^{1}$, when the set of
variables that are conditioned on contains $\mathbf{L}$.

\begin{lemma}
\label{lemma:indep_equiv} Let $U,V\in \mathbf{V}(\mathcal{H}^{1})$ and $%
\mathbf{L}\subset \mathbf{W}\subset \mathbf{V}(\mathcal{H}^{1})$. Then $%
U\perp _{\mathcal{H}^{0}}V\mid \mathbf{W}$ if and only if $U\perp _{\mathcal{%
H}^{1}}V\mid \mathbf{W}. $
\end{lemma}

Next we note that steps 1 and 2 of Definition \ref{def:H} also ensure that $%
A-Y$ cuts in $\mathcal{H}^{1}$ intersect neither $\forb(A,Y,\mathcal{G})$
nor $\mathbf{N}^{c}$, while step 3 ensures that all $A-Y$ cuts in $\mathcal{H%
}^{1}$ are supersets of $\mathbf{L}$. This, together with the preceding
discussion, suggests that $A-Y$ cuts in $\mathcal{H}^{1}$ should coincide
with $\mathbf{L}-\mathbf{N}$ dynamic adjustment sets in $\mathcal{G}$.
Proposition \ref{prop:equiv_dyn} below establishes that this is indeed true
when $\mathbf{N}\subset \an_{\mathcal{G}}(\{A,Y\}\cup \mathbf{L})$.

\begin{proposition}
\label{prop:equiv_dyn} $\:$

\begin{enumerate}
\item If $(\mathbf{L,N})$ is an admissible pair with respect to $A,Y$ in $%
\mathcal{G}$ $\ $then $A$ and $Y$ are not adjacent in $\mathcal{H}^{1}$.

\item If $\mathbf{Z}$ is an $A-Y$ cut in $\mathcal{H}^{1}$ then $\mathbf{Z}$
is an $\mathbf{L}-\mathbf{N}$ dynamic adjustment set with respect to $A,Y$
in $\mathcal{G}$.

\item If $\mathbf{Z}\subset \an_{\mathcal{G}}(\{A,Y\}\cup \mathbf{L})$ then $%
\mathbf{Z}$ is an $\mathbf{L}-\mathbf{N}$ dynamic adjustment set with
respect to $A,Y$ in $\mathcal{G}$ if and only if $\mathbf{Z}$ is an $A-Y$
cut in $\mathcal{H}^{1}$.

\item $\mathbf{Z}$ is a minimal $\mathbf{L}-\mathbf{N}$ dynamic adjustment
set with respect to $A,Y$ in $\mathcal{G}$ in $\mathcal{G}$ if and only if $%
\mathbf{Z}$ is a minimal $A-Y$ cut in $\mathcal{H}^{1}$.

\item $\mathbf{Z}$ is a minimum $\mathbf{L}-\mathbf{N}$ dynamic adjustment
set with respect to $A,Y$ in $\mathcal{G}$ if and only if $\mathbf{Z}$ is a
minimum $A-Y$ cut in $\mathcal{H}^{1}$.
\end{enumerate}
\end{proposition}

Next, we provide several examples illustrating the construction of $\mathcal{%
H}^{1}$. It is easy to check that in each of our examples $(\mathbf{L},%
\mathbf{N})$ is an admissible pair with respect to $A,Y$ in $\mathcal{G}$.

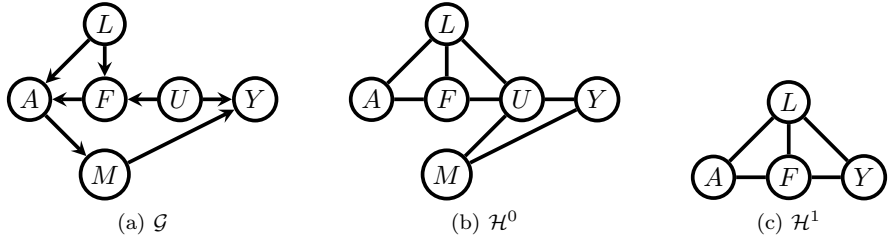
\begin{figure}[ht!]
\centering
\subfloat[$\mathcal{G}$]{
\begin{tikzpicture}[>=stealth, node distance=1cm,
pre/.style={->,>=stealth,ultra thick,line width = 1.4pt}]
  \begin{scope}
    \tikzstyle{format} = [circle, inner sep=2pt,draw, thick, circle, line width=1.4pt, minimum size=3mm]
    \node[format] (A) {$A$};
        \node[format, right of=A] (F) {$F$};
       \node[format, right of=F] (U) {$U$};
              \node[format, below of=F] (M) {$M$};
            \node[format, right of=U] (Y) {$Y$};
    \node[format, above of=F] (L) {$L$};

                 \draw (F) edge[pre, black] (A);
                 \draw (L) edge[pre, black] (A);
                 \draw (L) edge[pre, black] (F);
                 \draw (U) edge[pre, black] (F);                                   				 \draw (U) edge[pre, black] (Y);                                  
				\draw (A) edge[pre, black] (M);                                  
				\draw (M) edge[pre, black] (Y);

  \end{scope} 
  \end{tikzpicture}
  } \qquad 
\subfloat[$\mathcal{H}^{0}$]{
  \begin{tikzpicture}[>=stealth, node distance=1cm,
pre/.style={>=stealth,ultra thick,line width = 1.4pt}]
  \begin{scope}
    \tikzstyle{format} = [circle, inner sep=2pt,draw, thick, circle, line width=1.4pt, minimum size=3mm]
  \node[format] (A) {$A$};
        \node[format, right of=A] (F) {$F$};
       \node[format, right of=F] (U) {$U$};
              \node[format, below of=F] (M) {$M$};
            \node[format, right of=U] (Y) {$Y$};
    \node[format, above of=F] (L) {$L$};

                 \draw (F) edge[pre, black] (A);
                 \draw (L) edge[pre, black] (A);
                 \draw (L) edge[pre, black] (F);
                  \draw (L) edge[pre, black] (U);
                 \draw (U) edge[pre, black] (F);                                   				 \draw (U) edge[pre, black] (Y);                                  
				\draw (M) edge[pre, black] (Y);                                  
         		\draw (M) edge[pre, black] (U);

  \end{scope} 
  \end{tikzpicture}
  } \qquad 
\subfloat[$\mathcal{H}^{1}$]{
\begin{tikzpicture}[>=stealth, node distance=1cm,
pre/.style={>=stealth,ultra thick,line width = 1.4pt}]
  \begin{scope}
    \tikzstyle{format} = [circle, inner sep=2pt,draw, thick, circle, line width=1.4pt, minimum size=3mm]
  \node[format] (A) {$A$};
        \node[format, right of=A] (F) {$F$};
            \node[format, right of=F] (Y) {$Y$};
    \node[format, above of=F] (L) {$L$};

                 \draw (F) edge[pre, black] (A);
                 \draw (L) edge[pre, black] (A);
                 \draw (L) edge[pre, black] (F);
                 \draw (L) edge[pre, black] (Y);								
                 \draw (F) edge[pre, black] (Y);

  \end{scope} 
  \end{tikzpicture}
  }
\caption{Example of the construction of the non-parametric adjustment
efficiency graph, where $\mathbf{L}=\lbrace L \rbrace$ and $\mathbf{N}%
=\lbrace A,M, L,F, Y \rbrace$. Here $\an_{\mathcal{G}}(\lbrace A,Y\rbrace
\cup \mathbf{L})=\mathbf{V}(\mathcal{G})$, $\forb(A,Y,\mathcal{G})=\lbrace
A,Y,M\rbrace$ and $\ignore(A,Y,\mathcal{G})=\lbrace M, U\rbrace$. }
\label{fig:all_equal}
\end{figure}

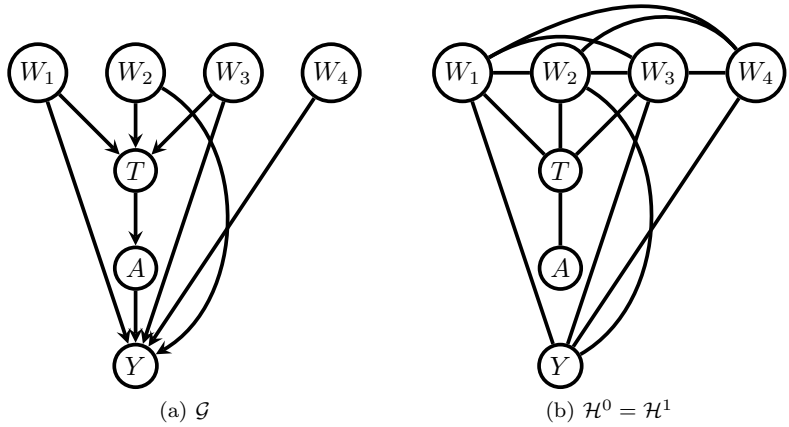
\begin{figure}[ht!]
\centering
\subfloat[$\mathcal{G}$]{
\begin{tikzpicture}[>=stealth, node distance=1.3cm,
pre/.style={->,>=stealth,ultra thick,line width = 1.4pt}]
  \begin{scope}
    \tikzstyle{format} = [circle, inner sep=2pt,draw, thick, circle, line width=1.4pt, minimum size=3mm]

    \node[format] (1) {$W_{1}$};
        \node[format,right of=1] (2) {$W_{2}$};
        \node[format,right of=2] (N) {$W_{3}$};
	\node[format,right of=N] (N1) {$W_{4}$};
                \node[format,below of=2] (L) {$T$};
            \node[format,below of=L] (A) {$A$};
             \node[format,below of=A] (Y) {$Y$};

                 \draw (1) edge[pre, black] (L);
          \draw (1) edge[pre, black] (Y);
                  \draw (2) edge[pre, black] (L);
          \draw (2) edge[pre, black, out=330, in=30] (Y);
          \draw (N) edge[pre, black] (L);
          \draw (N) edge[pre, black] (Y);
           \draw (A) edge[pre, black] (Y);
          \draw (L) edge[pre, black] (A);
                    \draw (N1) edge[pre, black] (Y);

  \end{scope} 
  \end{tikzpicture}
  } \qquad 
\subfloat[$\mathcal{H}^{0}=\mathcal{H}^{1}$]{
\begin{tikzpicture}[>=stealth, node distance=1.3cm,
pre/.style={>=stealth,ultra thick,line width = 1.4pt}]
  \begin{scope}
    \tikzstyle{format} = [circle, inner sep=2pt,draw, thick, circle, line width=1.4pt, minimum size=3mm]

    \node[format] (1) {$W_{1}$};
        \node[format,right of=1] (2) {$W_{2}$};
        \node[format,right of=2] (N) {$W_{3}$};
	\node[format,right of=N] (N1) {$W_{4}$};
                \node[format,below of=2] (L) {$T$};
            \node[format,below of=L] (A) {$A$};
             \node[format,below of=A] (Y) {$Y$};

                 \draw (1) edge[pre, black] (L);
          \draw (1) edge[pre, black] (Y);
                  \draw (2) edge[pre, black] (L);
          \draw (2) edge[pre, black, out=330, in=30] (Y);
          \draw (N) edge[pre, black] (L);
          \draw (N) edge[pre, black] (Y);
           \draw (1) edge[pre, black] (2);
            \draw (2) edge[pre, black] (N);
           \draw (N) edge[pre, black] (N1);
           
          \draw (1) edge[pre, black, out=30, in=150] (N);
             \draw (1) edge[pre, black, out=30, in=130] (N1);
          \draw (2) edge[pre, black, out=45, in=130] (N1);
          \draw (L) edge[pre, black] (A);
                    \draw (N1) edge[pre, black] (Y);

  \end{scope} 
  \end{tikzpicture}
  } 
\caption{Example of the construction of the non-parametric adjustment
efficiency graph, where $\mathbf{L}=\emptyset$ and $\mathbf{N}=%
\mathbf{V}(\mathcal{G})$. Here $\an_{\mathcal{G}}(\lbrace A,Y\rbrace \cup 
\mathbf{L})=\mathbf{V}(\mathcal{G})$, $\forb(A,Y,\mathcal{G})=\lbrace
A,Y\rbrace$ and $\ignore(A,Y,\mathcal{G})=\emptyset$. }
\label{fig:small_size}
\end{figure}

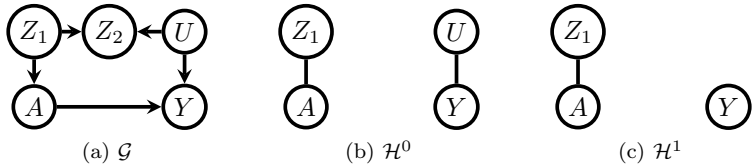
\begin{figure}[ht!]
\centering
\subfloat[$\mathcal{G}$]{
\begin{tikzpicture}[>=stealth, node distance=1cm,
pre/.style={->,>=stealth,ultra thick,line width = 1.4pt}]
  \begin{scope}
    \tikzstyle{format} = [circle, inner sep=2pt,draw, thick, circle, line width=1.4pt, minimum size=3mm]
    \node[format] (A) {$A$};
        \node[right of=A] (O) {};
        \node[format, right of=O] (Y) {$Y$};
               \node[format, above of=A] (Z1) {$Z_1$};
             \node[format ,above of=Y] (U) {$U$};
            \node[format, right of=Z1] (Z2) {$Z_2$};
                 \draw (Z1) edge[pre, black] (A);
                \draw (Z1) edge[pre, black] (Z2);
                        \draw (U) edge[pre, black] (Y);
               \draw (U) edge[pre, black] (Z2);
                              \draw (A) edge[pre, black] (Y);

  \end{scope} 
  \end{tikzpicture}
  } \qquad 
\subfloat[$\mathcal{H}^{0}$]{
  \begin{tikzpicture}[>=stealth, node distance=1cm,
pre/.style={>=stealth,ultra thick,line width = 1.4pt}]
  \begin{scope}
    \tikzstyle{format} = [circle, inner sep=2pt,draw, thick, circle, line width=1.4pt, minimum size=3mm]
    \node[format] (A) {$A$};
        \node[right of=A] (O) {};
        \node[format, right of=O] (Y) {$Y$};
               \node[format, above of=A] (Z1) {$Z_1$};
             \node[format ,above of=Y] (U) {$U$};

                 \draw (Z1) edge[pre, black] (A);

                        \draw (U) edge[pre, black] (Y);

  \end{scope} 
  \end{tikzpicture}
  } \qquad 
\subfloat[$\mathcal{H}^{1}$]{
\begin{tikzpicture}[>=stealth, node distance=1cm,
pre/.style={>=stealth,ultra thick,line width = 1.4pt}]
  \begin{scope}
 \tikzstyle{format} = [circle, inner sep=2pt,draw, thick, circle, line width=1.4pt, minimum size=3mm]
    \node[format] (A) {$A$};
        \node[right of=A] (O) {};
        \node[format, right of=O] (Y) {$Y$};
               \node[format, above of=A] (Z1) {$Z_1$};

                 \draw (Z1) edge[pre, black] (A);

  \end{scope} 
  \end{tikzpicture}
  }
\caption{Example of the construction of the non-parametric adjustment
efficiency graph, where $\mathbf{L}=\emptyset$ and $\mathbf{N}=\lbrace A,Y,
Z_{1},Z_{2}\rbrace$. Here $\an_{\mathcal{G}}(\lbrace A,Y\rbrace \cup \mathbf{%
L})=\lbrace A,Y, Z_{1}, U \rbrace$, $\forb(A,Y,\mathcal{G})=\lbrace
A,Y\rbrace$ and $\ignore(A,Y,\mathcal{G})=\lbrace U\rbrace$. }
\label{fig:no_opt}
\end{figure}

\begin{figure}[ht!]
\centering
\subfloat[$\mathcal{G}$]{
\begin{tikzpicture}[>=stealth, node distance=1cm,
pre/.style={->,>=stealth,ultra thick,line width = 1.4pt}]
  \begin{scope}
    \tikzstyle{format} = [circle, inner sep=2pt,draw, thick, circle, line width=1.4pt, minimum size=3mm]
    \node[format] (A) {$A$};
        \node[format, right of=A] (Y) {$Y$};
        \node[format, right of=Y] (U) {$U$};
             \node[format, below of=A] (L) {$L$};
               \node[right of=L] (O) {};
                 \node[format, right of=O] (F) {$F$};
                 \draw (L) edge[pre, black] (A);
               \draw (U) edge[pre, black] (F);
            \draw (U) edge[pre, black] (Y);
                 \draw (A) edge[pre, black] (Y);

  \end{scope} 
  \end{tikzpicture}
  } \qquad 
\subfloat[$\mathcal{H}^{0}$]{
  \begin{tikzpicture}[>=stealth, node distance=1cm,
pre/.style={>=stealth,ultra thick,line width = 1.4pt}]
  \begin{scope}
    \tikzstyle{format} = [circle, inner sep=2pt,draw, thick, circle, line width=1.4pt, minimum size=3mm]
    \node[format] (A) {$A$};
        \node[format, right of=A] (Y) {$Y$};
        \node[format, right of=Y] (U) {$U$};
             \node[format, below of=A] (L) {$L$};
               \node[right of=L] (O) {};

                 \draw (L) edge[pre, black] (A);

            \draw (U) edge[pre, black] (Y);

  \end{scope} 
  \end{tikzpicture}
  } \qquad 
\subfloat[$\mathcal{H}^{1}$]{
\begin{tikzpicture}[>=stealth, node distance=1cm,
pre/.style={>=stealth,ultra thick,line width = 1.4pt}]
  \begin{scope}
 \tikzstyle{format} = [circle, inner sep=2pt,draw, thick, circle, line width=1.4pt, minimum size=3mm]
    \node[format] (A) {$A$};
        \node[format, right of=A] (Y) {$Y$};
             \node[format, below of=A] (L) {$L$};
               \node[right of=L] (O) {};

                 \draw (L) edge[pre, black] (A);
                     \draw (L) edge[pre, black] (Y);

  \end{scope} 
  \end{tikzpicture}
  }
\caption{Example of the construction of the non-parametric adjustment
efficiency graph, where $\mathbf{L}=\lbrace L \rbrace$ and $\mathbf{N}%
=\lbrace A,Y,L,F\rbrace$. Here $\an_{\mathcal{G}}(\lbrace A,Y\rbrace \cup 
\mathbf{L})=\lbrace A,Y,L,U\rbrace$, $\forb(A,Y,\mathcal{G})=\lbrace
A,Y\rbrace$ and $\ignore(A,Y,\mathcal{G})=\lbrace U\rbrace$. }
\label{fig:only_min}
\end{figure}
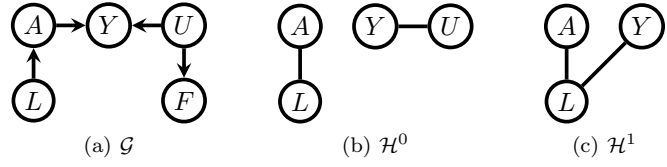

\FloatBarrier We will now define a binary relation in the class of $A-Y$
 cuts in $\mathcal{H}^{1}$ that will aid us in the
construction of our proposed graphical criteria for characterizing optimal $%
\mathbf{L}-\mathbf{N}$ dynamic adjustment sets. For $A-Y$ cuts $\mathbf{Z}%
_{1},\mathbf{Z}_{2}$ in $\mathcal{H}^{1}$, we define $\mathbf{Z}_{1}\unlhd _{%
\mathcal{H}^{1}}\mathbf{Z}_{2}$ if and only if $Y\perp _{\mathcal{H}^{1}}%
\mathbf{Z}_{2}\setminus \mathbf{Z}_{1}\mid \mathbf{Z}_{1}$ and $A\perp _{%
\mathcal{H}^{1}}\mathbf{Z}_{1}\setminus \mathbf{Z}_{2}\mid \mathbf{Z}_{2}.$
For example, in Figure \ref{fig:small_size} b), $\mathbf{Z}%
_{1}=\{W_{1},W_{2},W_{3}\}$ and $\mathbf{Z}_{2}=\{T\}$ are $A-Y$ cuts and $%
\mathbf{Z}_{1}\unlhd _{\mathcal{H}^{1}}\mathbf{Z}_{2}$. \cite{lattice}
showed $\unlhd _{\mathcal{H}^{1}}$ is a partial order in the class of
minimal (minimum) $A-Y$ cuts in $\mathcal{H}^{1}.$

Our next proposition and Proposition \ref{prop:compare_adj} entail that $%
\mathbf{Z}_{1}\unlhd _{\mathcal{H}^{1}}\mathbf{Z}_{2}$ implies $\mathbf{Z}%
_{1}\preceq_{\mathbf{L}}\mathbf{Z}_{2}$ for $\mathbf{L}-\mathbf{N}$ dynamic
adjustment sets $\mathbf{Z}_{1},\mathbf{Z}_{2}$ that are subsets of $\mathbf{V}%
\left( \mathcal{H}^{1}\right)$.

\begin{proposition}
\label{prop:order_implies_indep} If $\mathbf{Z}_{1}$ and $\mathbf{Z}_{2}$
are $\mathbf{L}-\mathbf{N}$ dynamic adjustment sets with respect to $A,Y$ in 
$\mathcal{G}$ such that $\mathbf{Z}_{1},\mathbf{Z}_{2}\subset \mathbf{V}%
\left( \mathcal{H}^{1}\right) $ and $\mathbf{Z}_{1}\unlhd _{\mathcal{H}^{1}}%
\mathbf{Z}_{2}$, then 
\begin{equation}
Y\perp \!\!\!\perp _{\mathcal{G}}\mathbf{Z}_{2}\setminus \mathbf{Z}_{1}\mid
A,\mathbf{Z}_{1}  \label{eq:indep_opt_y}
\end{equation}%
and 
\begin{equation}
A\perp \!\!\!\perp _{\mathcal{G}}\mathbf{Z}_{1}\setminus \mathbf{Z}_{2}\mid 
\mathbf{Z}_{2}.  \label{eq:indep_opt_a}
\end{equation}
\end{proposition}

Theorems 1 and 2 of \cite{lattice} imply that the set of all minimal
(minimum) $A-Y$ cuts in $\mathcal{H}^{1}$ are lattices with respect to $%
\unlhd _{\mathcal{H}^{1}}$with the infimum between two minimal (minimum) $%
A-Y $ cuts in $\mathcal{H}^{1}$ given by 
\begin{equation*}
\mathbf{Z}_{1}\wedge _{\mathcal{H}^{1}}\mathbf{Z}_{2}\equiv \partial _{%
\mathcal{H}^{1}}\left\{ \cc(\mathbf{Z}_{1}\cup \mathbf{Z}_{2},Y,\mathcal{H}%
^{1})\right\},
\end{equation*}%
which is a subset of $\mathbf{Z}_{1}\cup \mathbf{Z}_{2}$. This result,
together with Propositions \ref{prop:compare_adj}, \ref{prop:equiv_dyn} and %
\ref{prop:order_implies_indep}, entails the following Proposition.

\begin{proposition}
\label{prop:inf_minimal} Assume that $(\mathbf{L,N})$ is an admissible pair
with respect to $A,Y$ in $\mathcal{G}$. Then, the set of all minimal
(minimum) $\mathbf{L}-\mathbf{N}$ dynamic adjustment sets is a lattice with
respect to $\unlhd _{\mathcal{H}^{1}}$. Specifically, if $\ \mathbf{Z}_{1}$
and $\mathbf{Z}_{2}$ are minimal (minimum) $\mathbf{L}-\mathbf{N}$ dynamic
adjustment sets, then $\mathbf{Z}_{1}\wedge _{\mathcal{H}^{1}}\mathbf{Z}_{2}$
is a minimal (minimum) $\mathbf{L}-\mathbf{N}$ dynamic adjustment set.
Furthermore, 
\begin{equation}
(\mathbf{Z}_{1}\wedge _{\mathcal{H}^{1}}\mathbf{Z}_{2})\preceq_{\mathbf{L}}%
\mathbf{Z}_{1}\text{ and }(\mathbf{Z}_{1}\wedge _{\mathcal{H}^{1}}\mathbf{Z}%
_{2})\preceq _{\mathbf{L}}\mathbf{Z}_{2}.  \label{eq:infimum}
\end{equation}
\end{proposition}

\cite{eff-adj} showed that $\preceq $ is not a total preorder in the class
of minimal static adjustment sets (see their Example 2). This, implies that $%
\preceq_{\mathbf{L}}$ is not a total preorder. Nevertheless, Proposition \ref%
{prop:inf_minimal} implies that given two minimal (minimum) $\mathbf{L-N}$
dynamic adjustment sets $\mathbf{Z}_{1}$ and $\mathbf{Z}_{2}$ there exists
another one, namely $\mathbf{Z}_{1}\wedge _{\mathcal{H}^{1}}\mathbf{Z}_{2},$
included in their union which satisfies $\left( \ref{eq:infimum}\right) $.
The set $\mathbf{Z}_{1}\wedge _{\mathcal{H}^{1}}\mathbf{Z}_{2}$ thus yields
an NP-$\mathbf{Z}$ estimator of $\chi _{\pi }\left( P;\mathcal{G}\right) $
with variance smaller than or equal to the variance of the NP-$\mathbf{Z}$
estimators of $\chi _{\pi }\left( P;\mathcal{G}\right) $ that adjust for $%
\mathbf{Z}_{1}$ or for $\mathbf{Z}_{2},$ under any $P\in \mathcal{M(G)}$.

Next, we define the following sets which are our candidates for the optimal,
optimal minimal and optimal minimum $\mathbf{L}-\mathbf{N}$ dynamic
adjustment sets respectively,%
\begin{equation*}
\mathbf{O}(A,Y,\mathbf{L},\mathbf{N},\mathcal{G})\equiv \nb_{\mathcal{H}%
^{1}}(Y),\text{ }\mathbf{O}_{min}(A,Y,\mathbf{L},\mathbf{N},\mathcal{G}%
)\equiv \partial _{\mathcal{H}^{1}}\left\{ \cc(\nb_{\mathcal{H}^{1}}(Y),A,%
\mathcal{G})\right\}
\end{equation*}%
and for $\mathcal{C}_{\mathcal{H}^{1}}^{\ast }(A,Y)\equiv \{\mathbf{Z}%
_{1},\dots ,\mathbf{Z}_{l}\}$ the class of all minimum $A-Y$ cuts in $%
\mathcal{H}^{1},$ we define 
\begin{equation*}
\mathbf{O}_{m}(A,Y,\mathbf{L},\mathbf{N},\mathcal{G})\equiv \mathbf{Z}%
_{1}\wedge _{\mathcal{H}^{1}}\mathbf{Z}_{2}\wedge _{\mathcal{H}^{1}}\dots
\wedge _{\mathcal{H}^{1}}\mathbf{Z}_{l}
\end{equation*}

For brevity henceforth we will write $\mathbf{O},\mathbf{O}_{min}$ and $%
\mathbf{O}_{m}$ instead of $\mathbf{O}(A,Y,\mathbf{L},\mathbf{N},\mathcal{G}%
),\mathbf{O}_{min}(A,Y,\mathbf{L},\mathbf{N},\mathcal{G})$ and $\mathbf{O}%
_{m}(A,Y,\mathbf{L},\mathbf{N},\mathcal{G})$ respectively. The set $\mathbf{O%
}$ is comprised of vertices adjacent to $Y$ in $\mathcal{H}^{1}$ and $%
\mathbf{O}_{min}$ is the set of vertices in $\mathbf{O}$ that have at least
one path to $A$ in $\mathcal{H}^{1}$ that does not intersect any other
vertices of $\mathbf{O}$. It can be shown that if $\mathbf{L}=\emptyset $
and $\mathbf{N}=\mathbf{V} $ then $\mathbf{O}=\pa_{\mathcal{G}}(\cn(A,Y,%
\mathcal{G}))\setminus \forb_{\mathcal{G}}(A,Y,\mathcal{G})$ and $\mathbf{O}%
_{min}$ is equal to the smallest subset of $\mathbf{O}$ that satisfies $%
A\perp \!\!\!\perp _{\mathcal{G}}\mathbf{O}\setminus \mathbf{O}_{min}\mid 
\mathbf{O}_{min}$. Thus, our definitions of $\mathbf{O}$ and $\mathbf{O}%
_{min}$ coincide with the ones in \cite{eff-adj} in the special case in
which $\mathbf{L}=\emptyset $ and $\mathbf{N}=\mathbf{V}$. If $\mathbf{N}=%
\mathbf{V}$, it can be shown that $\mathbf{O}(A,Y,\mathbf{L},\mathbf{N},%
\mathcal{G})$ coincides with $\mathbf{O}(A,Y,\mathbf{L},\mathcal{G})$ as
defined in Section \ref{sec:compare}.

As for $\mathbf{O}_{m}$, we note that there are graphs for which the number
of minimum $A-Y$ cuts is exponential in the number of vertices in the graph,
yet although $\mathbf{O}_{m}$ is the infimum over all minimum $A-Y$ cuts,
its computation does not require enumeration of all the cuts. In fact, in
Section \ref{sec:algos} we provide a polynomial time algorithm to compute $%
\mathbf{O}_{m}$. We also provide a polynomial time algorithm to compute $%
\mathbf{O}_{min}$. On the other hand, one can trivially compute $\mathbf{O}$
in polynomial time by checking which variables are neighbors of $Y$ in $%
\mathcal{H}^{1}$. The following Theorem establishes that when $(\mathbf{L,N})
$ is an admissible pair, $\mathbf{O}_{\min }$ and $\mathbf{O}_{m}$ are the
optimal minimal and minimum $\mathbf{L}-\mathbf{N}$ dynamic adjustment sets
respectively. In addition $\mathbf{O}$ is a globally optimum $\mathbf{L}-%
\mathbf{N}$ dynamic adjustment set provided $\mathbf{N}\subset \an_{\mathcal{%
G}}(\{A,Y\}\cup \mathbf{L})$ or $\mathbf{N}=\mathbf{V}$.

\begin{theorem}
\label{theo:main_indep} Assume that $(\mathbf{L,N})$ is an admissible pair
with respect to $A,Y$ in $\mathcal{G}$. Then

\begin{enumerate}
\item $\mathbf{O}$ is an $\mathbf{L}-\mathbf{N}$ dynamic adjustment set with
respect to $A,Y$ in $\mathcal{G}$. In addition, if $\ \mathbf{N}\subset \an_{%
\mathcal{G}}(\{A,Y\}\cup \mathbf{L})$ or if $\ \mathbf{N}=\mathbf{V}$, then
for any other $\mathbf{L}-\mathbf{N}$ dynamic adjustment set $\mathbf{Z}$ it
holds that 
\begin{equation*}
Y\perp \!\!\!\perp _{\mathcal{G}}\mathbf{Z}\setminus \mathbf{O}\mid A,%
\mathbf{O}\quad \text{and}\quad A\perp \!\!\!\perp _{\mathcal{G}}\mathbf{O}%
\setminus \mathbf{Z}\mid \mathbf{Z}.
\end{equation*}%
Consequently $\mathbf{O}\preceq_{\mathbf{L}}\mathbf{Z}.$

\item $\mathbf{O}_{min}$ is a minimal $\mathbf{L}-\mathbf{N}$ dynamic
adjustment set with respect to $A,Y$ in $\mathcal{G}$. In addition, for any
other minimal $\mathbf{L}-\mathbf{N}$ dynamic adjustment set $\mathbf{Z}$ it
holds that 
\begin{equation*}
Y\perp \!\!\!\perp _{\mathcal{G}}\mathbf{Z}\setminus \mathbf{O}_{min}\mid A,%
\mathbf{O}_{min}\quad \text{and}\quad A\perp \!\!\!\perp _{\mathcal{G}}%
\mathbf{O}_{min}\setminus \mathbf{Z}\mid \mathbf{Z}.
\end{equation*}%
Consequently $\mathbf{O}_{\min }\preceq_{\mathbf{L}} \mathbf{Z}.$

\item $\mathbf{O}_{m}$ is a minimum $\mathbf{L}-\mathbf{N}$ dynamic
adjustment set with respect to $A,Y$ in $\mathcal{G}$. In addition, for any
other minimum $\mathbf{L}-\mathbf{N}$ dynamic adjustment set $\mathbf{Z}$ it
holds that 
\begin{equation*}
Y\perp \!\!\!\perp _{\mathcal{G}}\mathbf{Z}\setminus \mathbf{O}_{m}\mid A,%
\mathbf{O}_{m}\quad \text{and}\quad A\perp \!\!\!\perp _{\mathcal{G}}\mathbf{%
O}_{m}\setminus \mathbf{Z}\mid \mathbf{Z}.
\end{equation*}%
Consequently $\mathbf{O}_{m}\preceq_{\mathbf{L}}\mathbf{Z}.$
\end{enumerate}
\end{theorem}

\begin{example}
\label{ex:all_equal} Consider the graphs in Figure \ref{fig:all_equal}.
There is only one $A-Y$ cut in $\mathcal{H}^{1}$ in Figure \ref%
{fig:all_equal} c). Thus, in this case, $\mathbf{O}=\mathbf{O}_{min}=\mathbf{%
O}_{m} =\lbrace L,F \rbrace$.
\end{example}

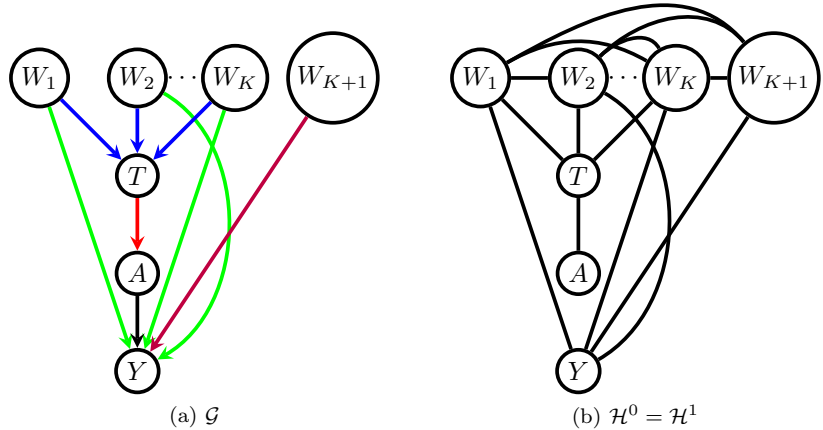
\begin{figure}[ht!]
\centering
\subfloat[$\mathcal{G}$]{
\begin{tikzpicture}[>=stealth, node distance=1.3cm,
pre/.style={->,>=stealth,ultra thick,line width = 1.4pt}]
  \begin{scope}
    \tikzstyle{format} = [circle, inner sep=2pt,draw, thick, circle, line width=1.4pt, minimum size=3mm]

    \node[format] (1) {$W_{1}$};
        \node[format,right of=1] (2) {$W_{2}$};
        \node[format,right of=2] (K) {$W_{K}$};
	\node[format,right of=K] (K1) {$W_{K+1}$};
                \node[format,below of=2] (L) {$T$};
            \node[format,below of=L] (A) {$A$};
             \node[format,below of=A] (Y) {$Y$};

        \path (2) -- node[auto=true]{$\cdots$} (K);
                 \draw (1) edge[pre, blue] (L);
          \draw (1) edge[pre, green] (Y);
                  \draw (2) edge[pre, blue] (L);
          \draw (2) edge[pre, green, out=330, in=30] (Y);
          \draw (K) edge[pre, blue] (L);
          \draw (K) edge[pre, green] (Y);
           \draw (A) edge[pre, black] (Y);
          \draw (L) edge[pre, red] (A);
                    \draw (K1) edge[pre, purple] (Y);

  \end{scope} 
  \end{tikzpicture}
  } \qquad 
\subfloat[$\mathcal{H}^{0}=\mathcal{H}^1$]{
\begin{tikzpicture}[>=stealth, node distance=1.3cm,
pre/.style={>=stealth,ultra thick,line width = 1.4pt}]
  \begin{scope}
    \tikzstyle{format} = [circle, inner sep=2pt,draw, thick, circle, line width=1.4pt, minimum size=3mm]

    \node[format] (1) {$W_{1}$};
        \node[format,right of=1] (2) {$W_{2}$};
        \node[format,right of=2] (K) {$W_{K}$};
	\node[format,right of=K] (K1) {$W_{K+1}$};
                \node[format,below of=2] (L) {$T$};
            \node[format,below of=L] (A) {$A$};
             \node[format,below of=A] (Y) {$Y$};

        \path (2) -- node[auto=true]{$\cdots$} (K);
                 \draw (1) edge[pre, black] (L);
          \draw (1) edge[pre, black] (Y);
                  \draw (2) edge[pre, black] (L);
          \draw (2) edge[pre, black, out=330, in=30] (Y);
          \draw (K) edge[pre, black] (L);
          \draw (K) edge[pre, black] (Y);
           \draw (1) edge[pre, black] (2);
           \draw (K) edge[pre, black] (K1);
           
          \draw (1) edge[pre, black, out=30, in=150] (K);
             \draw (1) edge[pre, black, out=30, in=130] (K1);
         \draw (2) edge[pre, black, out=45, in=120] (K);
          \draw (2) edge[pre, black, out=45, in=130] (K1);
          \draw (L) edge[pre, black] (A);
                    \draw (K1) edge[pre, black] (Y);

  \end{scope} 
  \end{tikzpicture}
  } 
\caption{Example of the construction of the non-parametric adjustment
efficiency graph, where $\mathbf{L}=\emptyset$ and $\mathbf{N}=%
\mathbf{V}$. Here, $K>2$, and the $\cdots$ between $W_{2}$ and $W_{K}$ stand
for the same pattern repeating itself. Thus, in $\mathcal{G}$, all vertices $%
W_{i}$, $i=1,\dots,K$ are parents of $T$ and of $Y$. In $\mathcal{H}^{0}=\mathcal{H}^{1}$ all vertices $W_{i}$, $i=1,\dots,K$ are adjacent to 
$T$, to $Y$ and to all $W_{j}$, $j=1,\dots,K+1$, $j\neq i$. In this example $
\an_{\mathcal{G}}(\lbrace A,Y\rbrace \cup \mathbf{L})=\mathbf{V}$, $\forb%
(A,Y,\mathcal{G})=\lbrace A,Y\rbrace$ and $\ignore(A,Y,\mathcal{G})=\emptyset
$. }
\label{fig:general_size}
\end{figure}

\begin{example}
\label{ex:general_size} Note that the DAG in Figure \ref{fig:general_size}
a) contains as a special case, when $K=3$, the DAG in Figure \ref%
{fig:small_size} a). In $\mathcal{H}^{1}$ in Figure \ref{fig:general_size}
b), the neighbors of $Y$ are $W_{1},\dots ,W_{K+1}$ and thus $\mathbf{O}%
=\{W_{1},\dots ,W_{K+1}\}$. Of the neighbors of $Y$ in $\mathcal{H}^{1}$ ,
only $W_{1},\dots ,W_{K}$ have paths to $A$ that don't intersect other
neighbors of $Y$. Thus $\mathbf{O}_{min}=\{W_{1},\dots ,W_{K}\}$. There is
only one minimum $A-Y$ cut in $\mathcal{H}^{1}$ and it is given by $\{T\}$.
Thus $\mathbf{O}_{m}=\{T\}$. For a large $K$, this example shows that the
optimal minimum $\mathbf{L}-\mathbf{N}$ dynamic adjustment set may have a
much smaller cardinality, in this case cardinality 1, than the optimal
minimal $\mathbf{L}-\mathbf{N}$ dynamic adjustment set, which in this case
has cardinality $K$. It can be shown that in this example $\mathbf{O}_{m}$
can be much less efficient than $\mathbf{O}_{min}$, and in turn $\mathbf{O}%
_{min}$ can be much less efficient than $\mathbf{O}$. Informally, $\mathbf{O}%
_{m}$ will be much less efficient than $\mathbf{O}_{min}$ when in Figure \ref%
{fig:general_size} a), the associations encoded in the green and red arrows
are strong and the associations encoded in the blue arrows are weak. $%
\mathbf{O}_{min}$ will be much less efficient than $\mathbf{O}$ when the
associations encoded in the green arrows are weak, and the associations
encoded in the blue, red and purple arrows are strong.
\end{example}

\begin{example}
\label{ex:no_opt} Consider the DAG in Figure \ref{fig:no_opt} a). Here $%
\mathbf{L}=\emptyset$, $\mathbf{N}=\lbrace A,Y,Z_{1},Z_{2}\rbrace$. Note
that $\mathbf{N}\neq \mathbf{V}(\mathcal{G})$ and $\an_{\mathcal{G}}(\lbrace
A,Y\rbrace \cup \mathbf{L}) \not\subset \mathbf{N}$ and hence the
assumptions needed in part 1) of Theorem \ref{theo:main_indep} for $\mathbf{O%
}$ to be the optimal $\mathbf{L}-\mathbf{N}$ dynamic adjustment set do not
hold.

All possible $\mathbf{L}-\mathbf{N}$ dynamic adjustment sets in $\mathcal{G}$
are $\mathbf{Z}^{\ast }=\emptyset $, ${\mathbf{Z}^{\ast \ast }}\mathbf{=}%
\left\{ Z_{1},Z_{2}\right\} $ and $\mathbf{Z}^{\ast \ast \ast }=\left\{
Z_{1}\right\} $. \cite{eff-adj} showed that $\mathbf{Z}^{\ast }\preceq $ $%
\mathbf{Z}^{\ast \ast \ast }$ and that no optimal static, and consequently
no optimal $\mathbf{L}-\mathbf{N}$ dynamic adjustment set, exists since
there are two distinct $P$ and $P^{\prime }$ in $\mathcal{M(G)}$ such that, for $%
\pi (A\mid \mathbf{L})=I_{1}(A)$, $\sigma _{\pi ,\mathbf{Z}^{\ast
}}(P)<\sigma _{\pi ,\mathbf{Z}^{\ast \ast }}(P)$ and $\sigma _{\pi ,\mathbf{Z%
}^{\ast }}^{2}(P^{\prime })>\sigma _{\pi ,\mathbf{Z}^{\ast \ast
}}^{2}(P^{\prime })$. However, by parts 2) and 3) of Theorem \ref%
{theo:main_indep}, $\mathbf{O}_{min}=\mathbf{O}_{m}=\emptyset $ is the
optimal minimum and minimal $\mathbf{L}-\mathbf{N}$ dynamic adjustment set.
\end{example}

\begin{example}
\label{ex:only_min} In the DAG of Figure \ref{fig:only_min}, $\mathbf{L}%
=\{L\}$, $\mathbf{N}=\{A,Y,L,F\}$. Here $\mathbf{N}\neq \mathbf{V}$ and $\an%
_{\mathcal{G}}(\{A,Y\}\cup \mathbf{L})\not\subset \mathbf{N}$ and hence the
assumptions needed in part 1) of Theorem \ref{theo:main_indep} for $\mathbf{O%
}=\{L\}$ to be the globally optimal $\mathbf{L}-\mathbf{N}$ dynamic
adjustment set do not hold. However, using Proposition \ref{prop:compare_adj}%
, it is easy to show that $\{L,F\}$ is the globally optimal $\mathbf{L}-%
\mathbf{N}$ dynamic adjustment set. This examples proves that the conditions
in part 1) of Theorem \ref{theo:main_indep} \ are sufficient but not
necessary for the existence of a globally optimal $\mathbf{L}-\mathbf{N}$
dynamic adjustment set. 
\end{example}

\section{Polynomial time algorithms to compute $\mathbf{O}_{min}$ and $%
\mathbf{O}_{m}$}

\label{sec:algos}

In this section we provide polynomial time algorithms to compute $\mathbf{O}%
_{m}$ and $\mathbf{O}_{min}$. We assume the availability of the following
sub-routines:

\begin{enumerate}
\item \texttt{disjointPaths}($A,Y,\mathcal{H}$). Computes a maximal number
of inner vertex disjoint paths between $A$ and $Y$. This can be done in $%
\mathcal{O}\left[ \{\#\mathbf{V}(\mathcal{H})\}^{1/2}\#\mathbf{E}(\mathcal{H}%
)\right] $ time using a maximum flow algorithm. See Corollary 7.1.5 in \cite%
{graph-book}. By Manger's Theorem (see Theorem 7.1.4 of \cite{graph-book}),
a routine \texttt{minCut}($A,Y,\mathcal{H}$) that computes the size of the
minimum $A-Y$ cut in an undirected graph $\mathcal{H}$ can also be
implemented in $\mathcal{O}\left[ \{\#\mathbf{V}(\mathcal{H})\}^{1/2}\#%
\mathbf{E}(\mathcal{H})\right] $ time.

\item \texttt{testExistAdj}($A,Y,\mathbf{L},\mathbf{N},\mathcal{G}$).
Returns $\true$ if and only if there exists an adjustment set $\mathbf{Z}$
with respect to $A,Y$ in $\mathcal{G}$ that satisfies $\mathbf{L}\subset 
\mathbf{Z}\subset \mathbf{N}$. This can be done in $\mathcal{O}\{\#\mathbf{V}%
(\mathcal{G})+\#\mathbf{E}(\mathcal{G})\}$ time using the \texttt{FINDADJ}
routine developed in \cite{vanAI}.
\end{enumerate}

\begin{algorithm}[ht!]
	\SetKwInOut{Input}{input}\SetKwInOut{Output}{output}
    \SetAlgoLined\DontPrintSemicolon
    \SetKwFunction{findOptMin}{findOptMin}
    \SetKwFunction{findIndPaths}{findIndPaths}
    \SetKwFunction{isInMinimum}{isInMinimum}
    \SetKwFunction{minCut}{minCut}
    \SetKwProg{proc}{procedure}{}{}
	\Input{An undirected graph $\mathcal{H}$ with vertex set $\mathbf{V}$, two non-adjacent vertices $A,Y \in \mathbf{V}$ and $V \in \mathbf V \setminus \{ A, Y \}$.}
	\Output{Boolean. True if there exists a minimum $A-Y$ cut $\mathbf{Z}$ with $V \in \mathbf{Z}$.}
	\proc{\isInMinimum{$V,A,Y,\mathcal{H}$}}
	{           
	$\mathbf{E}^{\prime}=\mathbf{E}(\mathcal{H})\cup \lbrace \lbrace A,V\rbrace,\lbrace V,Y\rbrace \rbrace$.\\
	$\mathcal{H}^{\prime}=(\mathbf{V}(\mathcal{H}), \mathbf{E}^{\prime})$\\
	$m_{1}= \# \texttt{minCut}(A, Y, \mathcal{H}^{\prime})$\\
    $m_{2}= \# \texttt{minCut}(A, Y, \mathcal{H})$\\
    \If{ $m_1 = m_2$}{
        \Return{$\true$}
    }
    \Else{ \Return{$\false$}}
	} 
		\caption{{\bf Subroutine to determine if a vertex is a member of a minimum $A-Y$ cut} }
			\label{algo:isinmin}
\end{algorithm}

\begin{algorithm}[ht!]
	\SetKwInOut{Input}{input}\SetKwInOut{Output}{output}
    \SetAlgoLined\DontPrintSemicolon
    \SetKwFunction{findOptMinimum}{findOptMinimum}
    \SetKwFunction{findIndPaths}{findIndPaths}
    \SetKwFunction{isInMinimum}{isInMinimum}
    \SetKwFunction{testExistsAdj}{testExistsAdj}
    \SetKwProg{proc}{procedure}{}{}
	\Input{$\mathcal{G}$ a DAG with vertex set $\mathbf{V}$ and vertices $A,Y \in \mathbf{V}$ such that $A\in \an_{\mathcal{G}}(Y)$.  $(\mathbf{L, N})$ an admissible pair with respect to $A,Y$ in $\mathcal{G}$.}
	\Output{$\mathbf{O}_{m}$.}
	\proc{\findOptMinimum{$A,Y,\mathcal{G}, \mathbf{N}, \mathbf{L}$}}
	{           
	    \If{\testExistsAdj{$A, Y, \mathbf{L}, \mathbf{N},\mathcal{G}$}}
	    {
	    construct $\mathcal{H}^{1}$\\
	    $\pi_1, \pi_2, \ldots, \pi_m =$  \texttt{disjointPaths}$(A,Y,\mathcal{H}^{1})$ \\
	    $\out = \emptyset$\\
	    \For{$i = 1, 2, \ldots, m$}{
	    $A - V_1 - V_2 - \ldots - V_{k_i} - Y = \pi_i$ \\
	        \For{  $ j = k_i, \ldots, 1$ }
	        { 
	            \If{ \isInMinimum{$V_j, \mathcal H^1$} }{
	               $\out = \out \cup \lbrace V_j \rbrace$ \\
	               \texttt{break}
	            }
	        }
	    }}
	    \Else{
	    $\out = \star$\\
	    }
	 \Return{$\out$}
	} 
		\caption{{\bf Algorithm to compute $\mathbf{O}_{m}$} }
			\label{algo:find_opt_min}
\end{algorithm}\FloatBarrier

\begin{algorithm}[ht!]
	\SetKwInOut{Input}{input}\SetKwInOut{Output}{output}
    \SetAlgoLined\DontPrintSemicolon
    \SetKwFunction{findOptMinimal}{findOptMinimal}
    \SetKwFunction{findIndPaths}{findIndPaths}
    \SetKwFunction{isInMinimum}{isInMinimum}
    \SetKwFunction{testExistsAdj}{testExistsAdj}
    \SetKwProg{proc}{procedure}{}{}
	\Input{$\mathcal{G}$ a DAG with vertex set $\mathbf{V}$ and vertices $A,Y \in \mathbf{V}$ such that $A\in \an_{\mathcal{G}}(Y)$.  $(\mathbf{L, N})$ an admissible pair with respect to $A,Y$ in $\mathcal{G}$.}
	\Output{$\mathbf{O}_{min}$.}
	\proc{\findOptMinimal{$A,Y,\mathcal{G}, \mathbf{N}, \mathbf{L}$}}
	{           
	    \If{\testExistsAdj{$A, Y, \mathbf{L}, \mathbf{N},\mathcal{G}$}}
	    {
	    construct $\mathcal{H}^{1}$\\
	    nb = $\nb_{\mathcal{H}^{1}}(Y)$\\
	    $\out = \emptyset$\\
   	    $\stack = \emptyset$\\
		$\visited = \emptyset$\\
		$\stack.$push($A$)\\
		\While{$\stack \neq \emptyset$}{
		$V=\stack.$pop()\\
		\uIf{$V\in \nb$ and $V\not\in \visited$}{
		$\out=\out \cup \lbrace V\rbrace $\\	
		$\visited=\visited \cup \lbrace V\rbrace $
		} \uElseIf{$V\not\in \visited$}
		{
		  $\visited = \visited \cup \lbrace V \rbrace$ \\
		  $\stack.$push($\nb_{\mathcal{H}^{1}}(V)$) \\
		}
		
		}
	    }
	    \Else{
	    $\out = \star$\\
	    }
	 \Return{$\out$}
	} 
		\caption{{\bf Algorithm to compute $\mathbf{O}_{min}$} }
			\label{algo:find_opt_minimal}
\end{algorithm}\FloatBarrier

In the Supplementary Material we prove the following results.

\begin{lemma}
\label{lemma:isinmin_correct} Algorithm \ref{algo:isinmin} outputs $\true$
if and only if there exists a minimum $A-Y$ cut $\mathbf{Z}$ in $\mathcal{H}$
with $V\in \mathbf{Z}$.
\end{lemma}

\begin{proposition}
\label{prop:algo_main_correct} Assume that $(\mathbf{L,N})$ is an admissible
pair with respect to $A,Y$ in $\mathcal{G}$. Then, the output of Algorithm %
\ref{algo:find_opt_min} is equal to $\mathbf{O}_{m}$. Furthermore, the
complexity of Algorithm is $\mathcal{O}\left[ \left\{ \#\mathbf{V}(\mathcal{G%
})\right\} ^{7/2}\right] $ .
\end{proposition}

\begin{proposition}
\label{prop:algo_minimal}Assume that $(\mathbf{L,N})$ is an admissible pair
with respect to $A,Y$ in $\mathcal{G}$. Then, the output of Algorithm \ref%
{algo:find_opt_minimal} is equal to $\mathbf{O}_{min}$. Furthermore, the
complexity of Algorithm is $\mathcal{O}\left[ \left\{ \#\mathbf{V}\left( 
\mathcal{G}\right) \right\} ^{2}\right].$
\end{proposition}

Algorithm \ref{algo:find_opt_minimal} is a simple modification of the depth
first search algorithm (see Section 8.2 of \cite{graph-book}). Its
complexity is dominated by the complexity of constructing $\mathcal{H}^{1}$,
which is $\mathcal{O}\left[ \left\{ \#\mathbf{V}\left( \mathcal{G}\right)
\right\} ^{2}\right]$.

\section{Discussion}

\label{sec:discussion}

In this paper we have shown that for $(\mathbf{L},\mathbf{N})$ an admissible
pair, a globally optimal $\mathbf{L}-\mathbf{N}$ dynamic adjustment set
exists when $\mathbf{N}\subset\an_{\mathcal{G}}(\lbrace A,Y\rbrace \cup 
\mathbf{L})$. We also noted that there are graphs that admit an $\mathbf{L}-%
\mathbf{N}$ dynamic adjustment set but with no globally optimal one.
However, in Example \ref{ex:only_min} we exhibited a graph such that $%
\mathbf{N}\neq \mathbf{V}$ and $\mathbf{N}\not \subset\an_{\mathcal{G}%
}(\lbrace A,Y\rbrace \cup \mathbf{L})$, but an optimal $\mathbf{L}-\mathbf{N}
$ dynamic adjustment set does exist. A complete characterization of the full
class of graphs under which a globally optimal $\mathbf{L}-\mathbf{N}$
dynamic adjustment set exists remains an open problem.

The results in this paper are for point interventions, that is, for
interventions on a single treatment vertex $A$. For multiple time dependent
static interventions, \cite{eff-adj} introduced the notion of a time
dependent adjustment set and provided graphical criteria to compare two time
dependent adjustment sets. They also showed that there exist graphs without
hidden variables in which no optimal time dependent adjustment set exists.
The extension of these results to general dynamic treatment regimes and
graphs with hidden variables is an interesting problem that warrants further
research.

Another related line of research is the derivation of semiparametric
efficient estimators of the policy value of a static point intervention in
graphical models. Unlike the non-parametric estimators considered in this
paper, semiparametric efficient estimators exploit all the information
encoded in the assumed causal graphical model. \cite{eff-adj} propose a
semiparametric efficient estimator for static point interventions and DAGs
without hidden variables. \cite{rohit} derived the semiparametric efficient
influence function of the policy value in special classes of DAGs with
hidden variables. The derivation of the semiparametric efficient influence
function of the policy value of a static or dynamic regime in an arbitrary
DAG with hidden variables in which the policy value is identified remains an
important open problem.

\section{Supplementary Material}

\subsection{Proofs of results in Section \protect\ref{sec:back}}

\begin{proof}[Proof of Proposition \ref{prop:equiv_dtr}] 

\medskip

1)
Assume first that $\mathbf{Z}$ is an $\mathbf{L}-\mathbf{N}$ dynamic adjustment set with respect to $A,Y$ in $\mathcal{G}$. Let $y\in\mathbb{R}$ and $a \in\mathcal{A}$. Taking $\pi(A\mid \mathbf{L})=I_{a}(A)$, it follows from the definition of $\mathbf{L}-\mathbf{N}$ dynamic adjustment set  that
\begin{equation}
E_{P}\left[ E_{P}\left\lbrace I_{(-\infty, y]}(Y)|A=a,\mathbf{L},\pa_{\mathcal{G}}\left( A\right) \right\rbrace
\right] =E_{P} \left[ E_{P}\left\lbrace I_{(-\infty, y]}(Y)|A=a,\mathbf{Z}\right\rbrace \right].
\label{eq:adj_dtr_1}
\end{equation} 
Note that
$$
E_{P}\left[ E_{P}\left\lbrace I_{(-\infty, y]}(Y)|A=a,\mathbf{L},\pa_{\mathcal{G}}\left( A\right) \right\rbrace %
\right] = E_{P}\left[ \frac{ I_{(-\infty, y]}(Y) I_{a}(A)}{f\left\lbrace A=a\mid \mathbf{L}, \pa_{\mathcal{G}}(A) \right\rbrace} \right].
$$
Since by assumption $\mathbf{L} \subset \nd_{\mathcal{G}}(A)$, by the Local Markov Property it holds that $A \ort \mathbf{L} \mid \pa_{\mathcal{G}}(A)$ under all $P \in \mathcal{M(G)}$. This implies
$$
f \left\lbrace A=a\mid \mathbf{L}, \pa_{\mathcal{G}}(A) \right\rbrace = f\left\lbrace A=a\mid  \pa_{\mathcal{G}}(A) \right\rbrace
$$
and hence
\begin{equation}
E_{P}\left[ E_{P}\left\lbrace I_{(-\infty, y]}(Y)|A=a,\mathbf{L},\pa_{\mathcal{G}}\left( A\right) \right\rbrace %
\right] = E_{P}\left[ \frac{ I_{(-\infty, y]}(Y) I_{a}(A)}{f\left\lbrace A=a\mid \pa_{\mathcal{G}}(A) \right\rbrace} \right]=E_{P}\left[ E_{P}\left\lbrace I_{(-\infty, y]}(Y)|A=a,\pa_{\mathcal{G}}\left( A\right) \right\rbrace %
\right] .
\label{eq:adj_dtr_2}
\end{equation}
Now \eqref{eq:adj_dtr_1}, \eqref{eq:adj_dtr_2} and the fact that by assumption $\mathbf{L}\subset \mathbf{Z} \subset \mathbf{N}$ imply that $\mathbf{Z}$ is an $\mathbf{L}-\mathbf{N}$ static adjustment set with respect to $A,Y$ in $\mathcal{G}$.

Assume now that $\mathbf{Z}$ is an $\mathbf{L}-\mathbf{N}$ static adjustment set with respect to $A,Y$ in $\mathcal{G}$. Take $y\in\mathbb{R}$.
We have to show that 
\begin{align*}
E_{P}\left( E_{\pi^{\ast}_{\mathbf{Z}}}\left[ E_{P}\left\lbrace I_{(-\infty,y]}(Y) \mid A, \mathbf{Z}\right\rbrace \mid \mathbf{Z}\right] \right)=E_{P}\left( E_{\pi^{\ast}}\left[ E_{P}\left\lbrace I_{(-\infty,y]}(Y) \mid A, \pa_{\mathcal{G}}(A),\mathbf{L}\right\rbrace \mid \pa_{\mathcal{G}}(A),\mathbf{L} \right] \right).
\end{align*}
For $a \in \mathcal{A}$ define
\begin{align*}
&B(a, \mathbf{Z})= E_{P}\left\lbrace I_{(-\infty,y]}(Y) \mid A=a, \mathbf{Z}\right\rbrace,
\\
&
\widetilde{B}(a, \pa_{\mathcal{G}}(A),\mathbf{L})= E_{P}\left\lbrace I_{(-\infty,y]}(Y) \mid A=a, \pa_{\mathcal{G}}(A),\mathbf{L}\right\rbrace.
\end{align*}
Note that
\begin{align*}
E_{P}\left( E_{\pi^{\ast}_{\mathbf{Z}}}\left[ E_{P}\left\lbrace I_{(-\infty,y]}(Y) \mid A, \mathbf{Z}\right\rbrace \mid \mathbf{Z}\right] \right) &=\sum\limits_{a\in\mathcal{A}} E_{P}\left[  \pi(a\mid \mathbf{L})  E_{P}\left\lbrace I_{(-\infty,y]}(Y) \mid A=a, \mathbf{Z}\right\rbrace \right]
\\
&
=
\sum\limits_{a\in\mathcal{A}} E_{P}\left[  \pi(a\mid \mathbf{L}) B(a, \mathbf{Z})   \right]
\\
&
=
\sum\limits_{a\in\mathcal{A}} E_{P}\left[  \pi(a\mid \mathbf{L}) E_{P}\left\lbrace B(a, \mathbf{Z})\mid \mathbf{L}\right\rbrace   \right]
.
\end{align*}
and
\begin{align*}
E_{P}\left( E_{\pi^{\ast}}\left[ E_{P}\left\lbrace I_{(-\infty,y]}(Y) \mid A, \pa_{\mathcal{G}}(A),\mathbf{L}\right\rbrace \mid \pa_{\mathcal{G}}(A),\mathbf{L} \right] \right) &=\sum\limits_{a\in\mathcal{A}} E_{P}\left[  \pi(a\mid \mathbf{L})  E_{P}\left\lbrace I_{(-\infty,y]}(Y) \mid A=a, \pa_{\mathcal{G}}(A),\mathbf{L}\right\rbrace \right]
\\
&
=
\sum\limits_{a\in\mathcal{A}} E_{P}\left[  \pi(a\mid \mathbf{L}) \widetilde{B}(a,  \pa_{\mathcal{G}}(A),\mathbf{L})   \right]
\\
&
=
\sum\limits_{a\in\mathcal{A}} E_{P}\left[  \pi(a\mid \mathbf{L}) E_{P}\left\lbrace \widetilde{B}(a,  \pa_{\mathcal{G}}(A),\mathbf{L}) \mid \mathbf{L}\right\rbrace  \right]
.
\end{align*}
Thus, to prove this part of the proposition, it suffices to show that
$$
E_{P}\left\lbrace \widetilde{B}(a,  \pa_{\mathcal{G}}(A),\mathbf{L}) \mid \mathbf{L}\right\rbrace=E_{P}\left\lbrace \widetilde{B}(a,  \pa_{\mathcal{G}}(A),\mathbf{L}) \mid \mathbf{L}\right\rbrace .
$$

Now, since $\pa_{\mathcal{G}}(A)\cup \mathbf{L}$ satisfies the back-door criterion, 
and since by assumption $\mathbf{Z}$ is an  $\mathbf{L}-\mathbf{N}$ static adjustment set, Corollary 2 of \cite{shpitser-adjustment} implies that
$$
\widetilde{B}(a,  \pa_{\mathcal{G}}(A),\mathbf{L}) = 
E\left\lbrace I_{(-\infty,y]}(Y_{a}) \mid \pa_{\mathcal{G}}(A),\mathbf{L} \right\rbrace
$$
and
$$
B(a,  \mathbf{Z}) = 
E\left\lbrace I_{(-\infty,y]}(Y_{a}) \mid \mathbf{Z} \right\rbrace,
$$
where $Y_{a}$ is a random variable with distribution equal to the marginal law of $Y$ when the vector $\mathbf{V}$ has a distribution given by
\begin{equation}
f_{a}\left( \mathbf{v}\right) = \delta_{a}(\mathbf{v}) \prod\limits_{V_{j}\in 
\mathbf{V}\setminus \{A\}}f\left\{ v_{j}\mid \pa_{\mathcal{G}%
}(v_{j})\right\} ,  \nonumber
\end{equation}%
with $\delta_{a}(\mathbf{v})$ being the indicator function that the coordinate of $\mathbf{v}$ corresponding to $A$ is equal to $a$.

Thus
$$
E\left\lbrace \widetilde{B}(a,  \pa_{\mathcal{G}}(A),\mathbf{L}) \mid \mathbf{L}\right\rbrace = E\left[ E\left\lbrace I_{(-\infty,y]}(Y_{a}) \mid \pa_{\mathcal{G}}(A),\mathbf{L} \right\rbrace \mid \mathbf{L}\right]= E\left\lbrace I_{(-\infty,y]}(Y_{a})  \mid \mathbf{L}\right\rbrace
$$
and, since $\mathbf{L}\subset\mathbf{Z}$,
$$
E\left\lbrace B(a,  \mathbf{Z}) \mid \mathbf{L}\right\rbrace = E\left[ E\left\lbrace I_{(-\infty,y]}(Y_{a}) \mid \mathbf{Z} \right\rbrace \mid \mathbf{L}\right]= E\left\lbrace I_{(-\infty,y]}(Y_{a})  \mid \mathbf{L}\right\rbrace.
$$
Hence
$$
E\left\lbrace B(a,  \mathbf{Z}) \mid \mathbf{L}\right\rbrace = E\left\lbrace \widetilde{B}(a,  \pa_{\mathcal{G}}(A),\mathbf{L}) \mid \mathbf{L}\right\rbrace.
$$
We conclude that $\mathbf{Z}$ is an $\mathbf{L}-\mathbf{N}$ dynamic adjustment set.

2) It suffices to show that if  $\mathbf{Z}$  is an $\mathbf{L}-\mathbf{N}$ dynamic adjustment set with respect to $A,Y$ in $\mathcal{G}$ then $\mathbf{Z} \cap \an_{\mathcal{G}}(\lbrace A,Y\rbrace \cup \mathbf{L})$ is an $\mathbf{L}-\mathbf{N}$ dynamic adjustment set too. 
Take $\mathbf{Z}$  an $\mathbf{L}-\mathbf{N}$ dynamic adjustment set.
By part 1), $\mathbf{Z}$ 
is an $\mathbf{L}-\mathbf{N}$ adjustment set. Then, by Theorem \ref{theo:van}, $Y \ort_{\mathcal{G}^{pbd}(A,Y)} A \mid \mathbf{Z}$ and $\mathbf{Z}\cap \forb(A,Y,\mathcal{G})=\emptyset$. By Lemma 1 from \cite{vanAI}, $Y \ort_{\mathcal{G}^{pbd}(A,Y)} A \mid \mathbf{Z}\cap  \an_{\mathcal{G}}(\lbrace A,Y\rbrace \cup \mathbf{L})$.  Using Theorem \ref{theo:van} again, we obtain that $\mathbf{Z}\cap  \an_{\mathcal{G}}(\lbrace A,Y\rbrace \cup \mathbf{L})$ is  an $\mathbf{L}-\mathbf{N}$ adjustment set.  Then part 1) of this proposition implies that 
$\mathbf{Z}\cap  \an_{\mathcal{G}}(\lbrace A,Y\rbrace \cup \mathbf{L})$ is  an $\mathbf{L}-\mathbf{N}$ dynamic adjustment set with respect to $A,Y$ in $\mathcal{G}$.

Finally, note that parts 3) and 4) follow immediately from part 1). This finishes the proof of the proposition.
\end{proof}

\subsection{Proofs of results in Section \protect\ref{sec:compare}}

To prove Lemmas \ref{lemma:supplementation} and \ref{lemma:deletion} we will
use the fact that for any $\mathbf{L}-\mathbf{N}$ dynamic adjustment set $%
\mathbf{Z}$ it holds that 
\begin{equation}
\psi _{P,\pi}\left( \mathbf{Z};\mathcal{G}\right) = \sum\limits_{a\in%
\mathcal{A}} \psi _{P,\pi,a}\left( \mathbf{Z};\mathcal{G}\right)
\label{eq:decomp_if}
\end{equation}
where 
\begin{equation*}
\psi _{P,\pi,a}\left( \mathbf{Z};\mathcal{G}\right)= I_{a}(A)\frac{\pi(a
\mid \mathbf{L})}{f(a\mid\mathbf{Z})}\lbrace Y- b(a,\mathbf{Z};P)\rbrace +
\pi(a\mid\mathbf{L}) b(a,\mathbf{Z};P) -E_{P}\left\lbrace \pi(a \mid \mathbf{%
L}) b(a,\mathbf{Z};P)\right\rbrace.
\end{equation*}
Define $\boldsymbol{\Psi}_{P,\pi}\left( \mathbf{Z};\mathcal{G}\right)=
\left( \psi _{P,\pi,a} \left( \mathbf{Z};\mathcal{G}\right) \right)_{a\in%
\mathcal{A}}. $

\begin{proof}[Proof of Lemma \protect\ref{lemma:supplementation}]
We show first that that $\mathbf{G} \cup \mathbf{B}$ is an $\mathbf{L}-\mathbf{N}$ dynamic adjustment set. The assumption  $A \ort_{\mathcal{G}} \mathbf{G} \mid \mathbf{B}$ implies
\begin{equation}
f\left(A\mid \mathbf{G,B}\right) = f(A\mid\mathbf{B}).
\label{eq:pi_eq}
\end{equation}
Then, for all $P\in\mathcal{M}(\mathcal{G})$,
\begin{align*}
&E_{P}\left\lbrace \frac{\q Y }{f(A\mid\mathbf{G,B})}\right\rbrace
= E_{P}\left\lbrace \frac{\q Y }{f(A\mid\mathbf{B})}\right\rbrace
= \chi_{\pi}(P,\mathcal{G}),
\end{align*}
where the last equality holds because $\mathbf{B}$ is, by assumption, an $\mathbf{L}-\mathbf{N}$ dynamic adjustment set. This proves that $\mathbf{G} \cup \mathbf{B}$ is an $\mathbf{L}-\mathbf{N}$ dynamic adjustment set. Using \eqref{eq:pi_eq} we obtain
\begin{equation}
E_{P}\left\lbrace \qa  b(a,\mathbf{G},\mathbf{B};P)\right\rbrace=E_{P}\left\lbrace \frac{I_{a}(a)\qa Y}{f(a\mid \mathbf{G},\mathbf{B})}\right\rbrace=E_{P}\left\lbrace\frac{I_{a}(a) \qa Y}{f(a\mid \mathbf{B})}\right\rbrace=E_{P}\left\lbrace \qa  b(a,\mathbf{B};P)\right\rbrace.
\label{eq:b_pi_eq}
\end{equation}

Write
\begin{eqnarray*}
\psi _{P,\pi,a}\left( \mathbf{B};\mathcal{G}\right)  &=&\frac{I_{a}(A)\qa Y}{f(a\mid\mathbf{B})}-\qa \left\lbrace \frac{I_{a}(A) }{f(a\mid \mathbf{B})}- 1\right\rbrace b(a,\mathbf{B};P)-E_{P}\left[\qa  b(a,\mathbf{B};P)\right]\\
&=& \frac{I_{a}(A) \qa Y}{f(a\mid\mathbf{G,B})}-\qa \left\lbrace \frac{I_{a}(A)}{f(a\mid\mathbf{G,B})}-1\right\rbrace b(a,\mathbf{G,B};P)\\
&+& \qa\left\lbrace \frac{I_{a}(A)}{f(a\mid\mathbf{G,B})}-1\right\rbrace \left\{ b(a,\mathbf{G},\mathbf{%
B};P)-b(a,\mathbf{B};P)\right\} -E_{P}\left\lbrace \qa  b(a,\mathbf{G},\mathbf{B};P)\right\rbrace \\
&=&\psi _{P,\pi,a}\left( \mathbf{G},\mathbf{B};\mathcal{G}\right) +%
\left\lbrace \frac{I_{a}(A)}{f(a\mid\mathbf{G,B})}-1\right\rbrace \qa \left\lbrace b(a,\mathbf{G},\mathbf{B};P)-b(a,\mathbf{B};P)\right\rbrace
\end{eqnarray*}%
where the second equality follows from $\left( \ref{eq:pi_eq}\right)$ and \eqref{eq:b_pi_eq}. 
Since
\begin{equation}
E_{P}\left\{ \psi _{P,\pi,a}\left( \mathbf{G,B};\mathcal{G}\right) g(%
A,\mathbf{G},\mathbf{B})\right\} =0\text{ for any }g\text{ such
that }E_{P}\left\lbrace g(A,\mathbf{G},\mathbf{B})|\mathbf{G},\mathbf{B}%
\right\rbrace =0  
\label{eq:uncorr}
\end{equation}%
and
$$
E_{P}\left[ \left\lbrace \frac{I_{a}(A)}{f(a \mid \mathbf{G},\mathbf{B})}-1\right\rbrace \qa \left\lbrace b(a,\mathbf{G},\mathbf{B};P)-b(a,\mathbf{B};P)\right\rbrace  \mid \mathbf{G},%
\mathbf{B}\right] =0
$$
we conclude that 
\begin{eqnarray*}
var_{P}\left\lbrace \psi _{P,\pi,a}\left( \mathbf{B};\mathcal{G}\right) \right\rbrace    
=var_{P}\left\lbrace \psi _{P,\pi,a}\left( \mathbf{G},\mathbf{B};\mathcal{G}%
\right) \right\rbrace +var_{P}\left[ \left\lbrace \frac{I_{a}(A)}{f(a \mid \mathbf{G},\mathbf{B})}-1\right\rbrace \qa \left\lbrace b(a,\mathbf{G},\mathbf{B};P)-b(a,\mathbf{B};P)\right\rbrace \right] .
\end{eqnarray*}%
Next note that
\begin{align}
&var_{P}\left[ \left\lbrace \frac{I_{a}(A)}{f(a \mid \mathbf{G},\mathbf{B})}-1\right\rbrace \qa\left\lbrace b(a,\mathbf{G},\mathbf{B};P)-b(a,\mathbf{B};P)\right\rbrace \right] =
\nonumber
\\
&
E_{P}\left[  \left\lbrace b(a,\mathbf{G,B};P)-b(a,\mathbf{B};P)\right\rbrace^{2}\qa^{2}var_{P}\left\lbrace \frac{I_{a}(A) }{f(a\mid\mathbf{G,B})}-1 \mid \mathbf{G,B}\right\rbrace \right] =
\nonumber
\\
&
E_{P}\left[  \left\lbrace b(a,\mathbf{G,B};P)-b(a,\mathbf{B};P)\right\rbrace^{2}\qa^{2}\left\lbrace \frac{1}{f(a\mid\mathbf{G,B})}-1 \right\rbrace \right]=
\nonumber
\\
\nonumber
&
E_{P}\left[  \left\lbrace b(a,\mathbf{G,B};P)-b(a,\mathbf{B};P)\right\rbrace^{2}\qa^{2}\left\lbrace \frac{1}{f(a\mid\mathbf{B})}-1 \right\rbrace \right]=
\\
&
E_{P}\left[ var_{P}\left\lbrace b(a,\mathbf{G,B};P)\mid \mathbf{B} \right\rbrace \qa^{2}\left\lbrace \frac{1}{f(a\mid\mathbf{B})}-1 \right\rbrace \right],
\label{eq:var_Sa}
\end{align}
where the last equality follows from 
\begin{align*}
b(a,\mathbf{B};P)=E_{P}\left( Y\mid A=a,\mathbf{%
B}\right)=E_{P}\left\lbrace E_{P}\left( Y\mid A=a,\mathbf{G},%
\mathbf{B}\right) \mid A=a,\mathbf{B}\right\rbrace 
&=E_{P}\left\lbrace
b(a,\mathbf{G,B})\mid A=a,\mathbf{B}\right\rbrace
\\
&=E_{P}\left\lbrace b(a,\mathbf{G,B})\mid \mathbf{B}\right\rbrace ,
\end{align*}
since $A \perp \!\!\!\perp _{\mathcal{G}}\mathbf{G}\mid \mathbf{B}$ by assumption.

Now, by \eqref{eq:decomp_if}, $\psi _{P,\pi}\left( \mathbf{Z};\mathcal{G}\right) = \mathbf{1}^{\top }\boldsymbol{\Psi}_{P,\pi}\left( \mathbf{Z};\mathcal{G}\right)$, where $\mathbf{1}$ is a vector of length $\# \mathcal{A}$ filled with ones.
Hence
$
var_{P}\left\lbrace \psi _{P,\pi}\left( \mathbf{B};\mathcal{G}\right) \right\rbrace = var_{P}\left\lbrace \mathbf{1}^{\top }\boldsymbol{\Psi}_{P,\pi}\left( \mathbf{B};\mathcal{G}\right)\right\rbrace$. Recall that $\mathbf{S}= \left( S_{a}\right)_{a\in\mathcal{A}}$, where
\begin{align*}
S_{a}\equiv \left\lbrace \frac{I_{a}(A)}{f(a\mid \mathbf{G},\mathbf{B})}-1\right\rbrace \qa \left\lbrace b(a,\mathbf{G},\mathbf{B};P)-b(a,\mathbf{B};P)\right\rbrace,
\end{align*} 
Since  by \eqref{eq:uncorr} it holds that $E_{P}\left( \mathbf{S} \mid \mathbf{G},\mathbf{B}\right)=\mathbf{0}$, we have
\begin{align*}
\sigma^{2}_{\pi,\mathbf{B}}(P)= var_{P}\left\lbrace \psi _{P,\pi}\left( \mathbf{B};\mathcal{G}\right) \right\rbrace = var_{P}\left\lbrace \mathbf{1}^{\top }\boldsymbol{\Psi}_{P,\pi}\left( \mathbf{B};\mathcal{G}\right) \right\rbrace &= var_{P}\left\lbrace \mathbf{1}^{\top }\boldsymbol{\Psi}_{P,\pi}\left( \mathbf{G,B};\mathcal{G}\right) \right\rbrace  + var_{P}\left\lbrace \mathbf{1}^{\top } \mathbf{S} \right\rbrace
\\
& = \sigma^{2}_{\pi,\mathbf{G},\mathbf{B}}(P) + \mathbf{1}^{\top }var_{P}(\mathbf{S})  \mathbf{1}.
\end{align*}
We already derived the expression for $var_{P}\left( S_{a}\right) $ in $
\left( \ref{eq:var_Sa}\right).$ Now, if $a\neq a^{\prime}$
\begin{eqnarray*}
&&cov_{P}\left( S_{a},S_{a^{\prime }}\right)  =
\\
&&E_{P}\left[
\left\{ \frac{I_{a}(A) \qa}{f(a\mid\mathbf{B})}%
-\qa\right\} \left\{ \frac{I_{a^{\prime}}(A) \qap}{f(a^{\prime}\mid\mathbf{B})}%
-\qap\right\}
 \left\{ b(a,\mathbf{G},%
\mathbf{B};P)-b(a,\mathbf{B};P)\right\} \left\{ b(a^{\prime},\mathbf{G},\mathbf{B};P)-b(a^{\prime},\mathbf{B}%
;P)\right\} \right]  =\\
&&E_{P}\left[ 
\left\{ \frac{I_{a}(A) \qa}{f(a\mid\mathbf{B})}%
-\qa\right\} \left\{ \frac{I_{a^{\prime}}(A) \qap}{f(a^{\prime}\mid\mathbf{B})}%
-\qap\right\} cov_{P}\left[ b(a,\mathbf{G},\mathbf{B};P),b(a^{\prime},\mathbf{G},%
\mathbf{B};P)|\mathbf{B},A	\right] \right] = \\
&&E_{P}\left[ 
\left\{ \frac{I_{a}(A) \qa}{f(a\mid\mathbf{B})}%
-\qa\right\} \left\{ \frac{I_{a^{\prime}}(A) \qap}{f(a^{\prime}\mid\mathbf{B})}%
-\qap\right\} cov_{P}\left[ b(a,\mathbf{G},\mathbf{B};P),b(a^{\prime},\mathbf{G},%
\mathbf{B};P)|\mathbf{B}	\right] \right]= \\
&&-E_{P}\left[ \qa \qap cov_{P}\left\lbrace b(a,\mathbf{G},\mathbf{B};P),b(a^{\prime},\mathbf{G},%
\mathbf{B};P)|\mathbf{B}	\right\rbrace \right] .
\end{eqnarray*} 

This finishes the proof Lemma \ref{lemma:supplementation}.
\end{proof}

\begin{proof}[Proof of Lemma \protect\ref{lemma:deletion}]
We first show that $\mathbf{G}$ is an $\mathbf{L}-\mathbf{N}$ dynamic adjustment set. The assumptions that $Y\perp \!\!\!\perp _{\mathcal{G}}\mathbf{B}%
\mid A,\mathbf{G}$ and $\mathbf{L}\subset\mathbf{G}$ imply that for all $P\in \mathcal{%
M}(\mathcal{G})$ 
\begin{eqnarray}
b(A,\mathbf{G,B};P) = E_{P}\left( Y\mid A,\mathbf{G},\mathbf{B}\right)   
=E_{P}\left( Y\mid A,\mathbf{G}\right)  
= b(A,\mathbf{G};P)  \label{eq:b_a}.
\end{eqnarray}%
Hence,
\begin{align*}
E_{P}\left[ E_{\pi^{\ast}_{\mathbf{G}}}\left\lbrace b(A,\mathbf{G};P)\mid \mathbf{G} \right\rbrace\right] =E_{P}\left[ E_{\pi^{\ast}_{\mathbf{G}}}\left\lbrace b(A,\mathbf{G},\mathbf{B};P)\mid \mathbf{G} \right\rbrace\right] =E_{P}\left[ E_{\pi^{\ast}_{\mathbf{G}\cup\mathbf{B}}}\left\lbrace b(A,\mathbf{G},\mathbf{B};P)\mid \mathbf{G}\cup\mathbf{B} \right\rbrace\right] =\chi_{\pi}(P;\mathcal{G}),
\end{align*}
where the first equality follows from \eqref{eq:b_a}, the second equality follows from the fact that, since 
$\pi^{\ast}_{\mathbf{G}}(A\mid \mathbf{G})\equiv\pi(A\mid \mathbf{L})$ and $\pi^{\ast}_{\mathbf{G}\cup\mathbf{B}}(A\mid \mathbf{G}\cup\mathbf{B})\equiv \pi(A\mid \mathbf{L})$ then $\pi^{\ast}_{\mathbf{G}}(A\mid \mathbf{G})=\pi^{\ast}_{\mathbf{G}\cup\mathbf{B}}(A\mid \mathbf{G}\cup\mathbf{B})$, and the third equality follows from the assumption that $\mathbf{G}\cup\mathbf{B} $ is an $\mathbf{L}-\mathbf{N}$ dynamic adjustment set. This shows that $\mathbf{G}$ is an $\mathbf{L}-\mathbf{N}$ dynamic adjustment set.

Next, write 
\[
var_{P}\left\lbrace \psi _{P,\pi,a}(\mathbf{G,B};\mathcal{G})\right\rbrace = var_{P}\left\lbrace E_{P}(\psi _{P,\pi,a}(\mathbf{G,B};%
\mathcal{G})\mid A,Y,\mathbf{G})\right\rbrace +E_{P}%
\left\lbrace var_{P}(\psi _{P,\pi,a}(\mathbf{G,B};\mathcal{G})\mid A
,Y,\mathbf{G})\right\rbrace .
\]%
Now 
\begin{align}
\nonumber
&E_{P}\left\lbrace \psi _{P,\pi,a}(\mathbf{G,B};\mathcal{G})\mid A,Y,\mathbf{G}\right\rbrace  =
\\
\nonumber
&E_{P}\left[ \frac{I_{a}(A) \qa}{f(a\mid\mathbf{G},\mathbf{B})} \lbrace Y - b(a,\mathbf{G};P)\rbrace + \qa b(a,\mathbf{G};P) - E_{P}\left \lbrace \qa b(a,\mathbf{G};P) \right\rbrace  \mid A,Y,\mathbf{G} \right]=
\\
\nonumber
&
I_{a}(A) \qa \lbrace Y - b(a,\mathbf{G};P)\rbrace E_{P}\left\lbrace \frac{1}{f(a\mid\mathbf{G},\mathbf{B})} \mid A,Y,\mathbf{G}\right\rbrace +\qa b(a,\mathbf{G};P) - E_{P}\left \lbrace \qa b(a,\mathbf{G};P) \right\rbrace = 
\\
\nonumber
&
I_{a}(A) \qa \lbrace Y - b(a,\mathbf{G};P)\rbrace  E_{P}\left\lbrace \frac{1}{f(a\mid\mathbf{G},\mathbf{B})} \mid A,\mathbf{G}\right\rbrace + \qa b(a,\mathbf{G};P) - E_{P}\left \lbrace \qa b(a,\mathbf{G};P) \right\rbrace =
\\
\nonumber
&
\frac{I_{a}(A) \qa}{f(a\mid\mathbf{G})} \lbrace Y - b(a,\mathbf{G};P)\rbrace  + \qa b(a,\mathbf{G};P) - E_{P}\left \lbrace \qa b(a,\mathbf{G};P) \right\rbrace =
\\
\nonumber
&
\psi _{P,\pi,a}(\mathbf{G};\mathcal{G}).
\end{align}
where the first equality follows from $\left( \ref{eq:b_a}\right)$, the
third equality follows from $Y\perp \!\!\!\perp _{\mathcal{G}}\mathbf{B}\mid A,\mathbf{G}$ and the fourth from Lemma 10 from \cite{eff-adj} which states that 
$$
E_{P}\left\lbrace \frac{1}{f(a\mid\mathbf{G},\mathbf{B})} \mid A,\mathbf{G}\right\rbrace =\frac{1}{f(a\mid\mathbf{G})} .
$$
On the
other hand, 
\begin{align*}
&var_{P}\left\lbrace \psi _{P,\pi,a}(\mathbf{G,B};\mathcal{G})\mid A%
,Y,\mathbf{G}\right\rbrace  =
\\
&var_{P}\left[ \frac{I_{a}(A) \qa}{f(a\mid \mathbf{G},\mathbf{B})} \lbrace Y - b(a,\mathbf{G};P) \rbrace + \qa b(a,\mathbf{G};P) - E_{P}\left \lbrace \qa b(a,\mathbf{G};P) \right\rbrace \mid A,Y,\mathbf{G}\right]  =
\\
&
I_{a}(A) \qa^{2} \lbrace Y - b(a,\mathbf{G};P)\rbrace ^{2}
var_{P}\left\{ \frac{1}{f(a\mid\mathbf{G,B})}\mid A=a,\mathbf{G}\right\} 
\end{align*}%
where the first equality follows from $\left( \ref{eq:b_a}\right) $ and the
second follows from $Y \ort_{\mathcal{G}}\mathbf{B}\mid A,\mathbf{G}.$ Thus 
\begin{align}
\nonumber
&E_{P}\left[ var_{P}\left\lbrace \psi _{P,\pi,a}(\mathbf{G,B};\mathcal{G})\mid A%
,Y,\mathbf{G}\right\rbrace \right] =
\\
&E_{P}\left[ \qa^{2}  f(a\mid\mathbf{G})var_{P}(Y\mid A= a,\mathbf{G})var_{P}\left\lbrace 
\frac{1}{f(a\mid\mathbf{G,B})}\mid A=a,\mathbf{G}%
\right\rbrace \right].
\label{eq:exp_cond_var}
\end{align}

Now, by \eqref{eq:decomp_if}, $\psi _{P,\pi}\left( \mathbf{Z};\mathcal{G}\right) = \mathbf{1}^{\top }\boldsymbol{\Psi}_{P,\pi}\left( \mathbf{Z};\mathcal{G}\right)$, where $\mathbf{1}$ is a vector of length $\# \mathcal{A}$ filled with ones.
Thus
$
\sigma _{\pi,\mathbf{G,B}}^{2}\left( P\right)=var_{P}\left\lbrace \psi _{P,\pi}\left( \mathbf{G}, \mathbf{B};\mathcal{G}\right) \right\rbrace = var_{P}\left\lbrace \mathbf{1}^{\top }\boldsymbol{\Psi}_{P,\pi}\left( \mathbf{G}, \mathbf{B};\mathcal{G}\right)\right\rbrace$. 
We then have
\begin{eqnarray*}
\sigma _{\pi,\mathbf{G,B}}^{2}\left( P\right)  &=&var_{P}\left[ E_{P}%
\left\lbrace \mathbf{1}^{\top} \boldsymbol{\Psi}_{P,\pi}\left( \mathbf{G,B};\mathcal{G}%
\right) |A,Y,\mathbf{G}\right\rbrace \right] +E_{P}\left[ var_{P}\left\lbrace 
\mathbf{1}^{\top}\boldsymbol{\Psi}_{P,\pi}\left( \mathbf{G,B};\mathcal{G}\right) |%
A,Y,\mathbf{G}\right\rbrace \right]  \\
&=&var_{P}\left\lbrace \mathbf{1}^{\top}\boldsymbol{\Psi}_{P,\pi}\left( \mathbf{G};\mathcal{G%
}\right) \right\rbrace +\mathbf{1}^{\top}E_{P}\left[ var_{P}\left\lbrace \boldsymbol{\Psi}
_{P,\pi}\left( \mathbf{G,B};\mathcal{G}\right) |A,Y,\mathbf{G}\right\rbrace %
\right] \mathbf{1} \\
&=&\sigma _{\pi,\mathbf{G}}^{2}\left( P\right) +\mathbf{1}^{\top}E_{P}\left[
var_{P}\left\lbrace \boldsymbol{\Psi}_{P,\pi}\left( \mathbf{G,B};\mathcal{G}\right) |%
A,Y,\mathbf{G}\right\rbrace \right] \mathbf{1}
\end{eqnarray*}%
We obtained an expression for $E_{P}\left[ var_{P}\left\lbrace \psi _{P,\pi,a}(\mathbf{G,B};\mathcal{G})\mid A%
,Y,\mathbf{G}\right\rbrace \right]$  in \eqref{eq:exp_cond_var}.
Using $\left( \ref{eq:b_a}\right) $, if $a\neq a^{\prime}$ we obtain
\begin{eqnarray*}
&&cov_{P}\left\lbrace \psi_{P,\pi,a}\left( \mathbf{G,B};\mathcal{G}\right) , \psi_{P,\pi,a^{\prime}}\left( \mathbf{G,B};\mathcal{G}\right) \mid A,Y, \mathbf{G}\right\rbrace   =
\\
&&cov_{P}\left[ \frac{I_{a}(A) \qa}{f(a\mid\mathbf{G},\mathbf{B})} \lbrace Y - b(a,\mathbf{G};P) \rbrace , \frac{I_{a^{\prime}}(A) \qap}{f(a^{\prime}\mid\mathbf{G},\mathbf{B})} \lbrace Y - b(a^{\prime},\mathbf{G};P)\rbrace |A,Y,\mathbf{G}%
\right]= 0,
\end{eqnarray*}
since $I_{a}(A)I_{a^{\prime}}(A)=0$.
This shows that
$$
\sigma _{\pi,\mathbf{G,B}}^{2}\left( P\right)  =\sigma _{\pi,\mathbf{G}}^{2}\left( P\right) + \sum\limits_{a\in\mathcal{A}} \left( E_{P}\left[ \pi^{2}(a\mid\mathbf{L})  f(a\mid \mathbf{G})var_{P}(Y\mid A= a,\mathbf{G})var_{P}\left\lbrace
\frac{1}{f(a\mid\mathbf{G,B})}\mid A=a,\mathbf{G}%
\right\rbrace \right] \right).
$$

This concludes the proof of Lemma \ref{lemma:deletion}.
\end{proof}

The proof of Proposition \ref{prop:opt_adj} below uses Proposition \ref%
{prop:equiv_dyn}, the proof of which can be found in the following section.

\begin{proof}[Proof of Proposition \ref{prop:opt_adj}]
Let $\mathbf{Z}$ be an $\mathbf{L}-\mathbf{N}$ dynamic adjustment set. Note that by Proposition \ref{prop:equiv_dtr}, $\mathbf{Z}$ is an $\mathbf{L}-\mathbf{N}$ static adjustment set.

We begin with the proof of  $A \ort_{\mathcal{G}} \mathbf{O}\setminus \mathbf{Z} \mid \mathbf{Z}$.  Consider a path $\pi$ in $\mathcal{G}$ between $A$ and a vertex $O \in \mathbf{O}\setminus \mathbf{Z}$. Note that since $\mathbf{O}$ does not contain descendants of $A$, if $\pi$ is directed then it enters $A$ through the back-door. Now, since $\mathbf{L} \subset\mathbf{O}$, $\mathbf{L} \subset\mathbf{Z}$ and $O \in \mathbf{O}\setminus \mathbf{Z}$, it holds that $O \in \mathbf{O} \setminus \mathbf{L}$. Hence, there exists a directed path from $O$ to $Y$, say $\delta$, such that all vertices in that path except for $O$ are members of $\forb(A,Y,\mathcal{G})$. The path from $A$ to $Y$ obtained by joining $\pi$ and $\delta$, say $\gamma$, is non-directed. Since $\mathbf{Z}$ is an $\mathbf{L}-\mathbf{N}$ adjustment set, by Corollary 1 from \cite{shpitser-adjustment}, it does not contain vertices in $\forb(A,Y,\mathcal{G})$ and it has to block $\gamma$. We conclude that $\pi$ must be blocked by $\mathbf{Z}$, which is what we wanted to show.

Next, we show that $Y \ort_{\mathcal{G}} \mathbf{Z}\setminus \mathbf{O} \mid A, \mathbf{O}$. Consider a path $\pi$ in $\mathcal{G}$ between $Y$ and a vertex $Z \in \mathbf{Z}\setminus \mathbf{O}$. Suppose for the sake of contradiction that $\pi$ is open given $ A, \mathbf{O}$. Now, if $\pi$ does not have colliders, by Lemma E.4 from \cite{perkovic}, it is blocked by   $A, \mathbf{O}\setminus \mathbf{L}$ and hence it is blocked by $A, \mathbf{O}$, which is a contradiction. Assume  then that $\pi$ has at least one collider. Let $C$ be the collider on $\pi$ that is closest to $Y$. Suppose first that in $\pi$ the edge containing $Y$ points out of $Y$.  Then $C$ is a descendant of $Y$. Since $\pi$ is open given $A,\mathbf{O}$, it follows that $C$ is an ancestor of either $A$ or $\mathbf{O}$. This implies that either $A$ or a vertex in $\mathbf{O}$ is a descendant of $Y$. This contradicts the fact that $A \in \an_{\mathcal{G}}(Y)$ and $\mathbf{O}\cap \de_{\mathcal{G}}(A)=\emptyset$. Suppose next that in $\pi$ the edge containing $Y$ points into $Y$.
Let $F$ bet the only fork on $\pi$ that lies between $C$ and $Y$. Since $\pi$ is open given $ A, \mathbf{O}$, $C$ is an ancestor of either $A$ or a vertex in $\mathbf{O}$. Since $F$ is an ancestor of $C$, it must be that $F$ is an ancestor of either $A$ or a vertex in $\mathbf{O}$. Since $\mathbf{O}\cap \de_{\mathcal{G}}(A)=\emptyset$, it follows that $F\not \in \de_{\mathcal{G}}(A)$. Thus, there exists a vertex on the sub-path of $\pi$ that goes from $F$ to $Y$ that is a member of $\mathbf{O}$. This implies that $\pi$ is closed given $A, \mathbf{O}$, which is a contradiction. Hence  $Y \ort_{\mathcal{G}} \mathbf{Z}\setminus \mathbf{O} \mid A, \mathbf{O}$ holds.

Finally, we prove that $\mathbf{O}$ is an $\mathbf{L}-\mathbf{N}$ dynamic adjustment set. Clearly, $\mathbf{L}\subset \mathbf{O}\subset \an_{\mathcal{G}}(\lbrace A,Y\rbrace \cup \mathbf{L})$. Note that from any vertex in $\mathbf{O}\setminus \mathbf{L}$ there exists a directed path in $\mathcal{G}$ to $Y$ that only intersects vertices in $\forb(A,Y,\mathcal{G})$. The definition of $\mathcal{H}^{1}$ then implies that $\mathbf{O}$ is exactly the set of neighbours of $Y$ in $\mathcal{H}^{1}$. Thus, $\mathbf{O}$ is an $A-Y$ cut in $\mathcal{H}^{1}$. It follows from part 2) of Proposition \ref{prop:equiv_dyn} that $\mathbf{O}$ is an $\mathbf{L}-\mathbf{N}$ dynamic adjustent set.
\end{proof}

\subsection{Proofs of results in Section \protect\ref{sec:new_graph}}

\begin{proof}[Proof fo Lemma \ref{lemma:indep_equiv}]
Assume first that $U \perp_{\mathcal{H}^{0}} V \mid \mathbf{W}$ holds. If no path between $U$ and $V$ in $\mathcal{H}^{1}$ exists, the result is trivial. Hence, assume there exists a path $\pi$ from $U$ to $V$ in $\mathcal{H}^{1}$. We will show that $\pi$ intersects $\mathbf{W}$. If all edges in $\pi$ are also present in $\mathcal{H}^{0}$ then clearly $\pi$ has to intersect a vertex in $\mathbf{W}$. Otherwise, if an edge, say $S-T$, in $\pi$ is not present in $\mathcal{H}^{0}$ then $S-T$ is of one of two types of edges:
(i) the edge goes from a vertex in $\mathbf{L}$ to either $A$ or $Y$,
or (ii) there exists a path from $S$ to $T$ in  $\mathcal{H}^{0}$ that goes only through vertices in $\ignore(A,Y,\mathcal{G})$. If there exists an edge $S-T$ in $\pi$ of type (i) then, since $\mathbf{L}\subset \mathbf{W}$, we conclude that $\pi$ is blocked by $\mathbf{W}$ in $\mathcal{H}^{1}$. Assume then that all edges in $S-T$ that are not present in $\mathcal{H}^{0}$ are of type (ii). Consider the path $\delta$ in $\mathcal{H}^{0}$ obtained from $\pi$ by replacing each edge $S-T$ in $\pi$ that is not present in $\mathcal{H}^{0}$ by the corresponding path in $\mathcal{H}^{0}$ from $S$ to $T$ that goes only through vertices in $\ignore(A,Y,\mathcal{G})$. Since by assumption $U \perp_{\mathcal{H}^{0}} V \mid \mathbf{W}$, the path $\delta$ has to intersect $\mathbf{W}$. Since $\mathbf{W}\subset \mathbf{V}(\mathcal{H}^{1}) $, we conclude that $\pi$ has to intersect $\mathbf{W}$. 

Assume next that $U \perp_{\mathcal{H}^{1}} V \mid \mathbf{W}$ holds. If no path between $U$ and $V$ in $\mathcal{H}^{0}$ exists, the result is trivial. Assume then that there exists a path $\pi$ from $U$ to $V$ in $\mathcal{H}^{0}$. We will show that $\pi$ intersects $\mathbf{W}$. If $\pi$ goes only through vertices that are not in $\ignore(A,Y,\mathcal{G})$ then $\pi$ is also a path in $\mathcal{H}^{1}$ and hence it intersects a vertex in $\mathbf{W}$. If $\pi$ intersects at least one vertex in $\ignore(A,Y,\mathcal{G})$, consider the path $\delta$ in $\mathcal{H}^{1}$ obtained from $\pi$ by removing all vertices in $\ignore(A,Y,\mathcal{G})$ and adding an edge between any pair of remaining vertices if they were connected in $\pi$ by a path going only through vertices $\ignore(A,Y,\mathcal{G})$. Since by assumption $U \perp_{\mathcal{H}^{1}} V \mid \mathbf{W}$, the path $\delta$ has to intersect $\mathbf{W}$. We conclude that $\pi$ has to intersect $\mathbf{W}$. 
\end{proof}

In order to prove Proposition \ref{prop:equiv_dyn}, we will need the
following lemmas.

\begin{lemma}
\label{lemma:cuts_contain} Let $\mathcal{G}$ be a DAG with vertex set $%
\mathbf{V} $ and assume that $\left( A,Y,\mathbf{L},\mathbf{N}\right) $
satisfy the inclusion conditions. Assume that $(\mathbf{L, N})$ is an
admissible pair with respect to $A,Y$ in $\mathcal{G}$. Then $\mathbf{V}(%
\mathcal{H}^{1})= \left\lbrace \an_{\mathcal{G}}(\lbrace A, Y\rbrace \cup 
\mathbf{L})\cap \mathbf{N} \right\rbrace \setminus \left\lbrace \forb(A,Y,%
\mathcal{G})\setminus \lbrace A,Y\rbrace \right\rbrace$. If $\mathbf{Z}$ is
an $A-Y$ cut in $\mathcal{H}^{1}$ then $\mathbf{L}\subset \mathbf{Z} \subset 
\mathbf{V}(\mathcal{H}^{1})$.
\end{lemma}

\begin{proof}
This follows immediately from the definition of $\mathcal{H}^{1}$.
\end{proof}

\begin{lemma}
\label{lemma:equiv} Let $\mathcal{G}$ be a DAG with vertex set $\mathbf{V} $
and assume that $\left( A,Y,\mathbf{L},\mathbf{N}\right) $ satisfy the
inclusion conditions.

\begin{enumerate}
\item If $\mathbf{Z}$ is an $A-Y$ cut in $\mathcal{H}^{1}$ then $\mathbf{Z}$
is a $\mathbf{L}-\mathbf{N}$ static adjustment set with respect to $A,Y$ in $%
\mathcal{G}$.

\item If $\mathbf{Z} \subset \an_{\mathcal{G}}(\lbrace A,Y\rbrace \cup 
\mathbf{L})$ then $\mathbf{Z}$ is an $\mathbf{L}-\mathbf{N}$ static
adjustment set with respect to $A,Y$ in $\mathcal{G}$ if and only if $%
\mathbf{Z}$ is an $A-Y$ cut in $\mathcal{H}^{1}$.

\item $\mathbf{Z}$ is a minimal $\mathbf{L}-\mathbf{N}$ static adjustment
set with respect to $A,Y$ in $\mathcal{G}$ if and only if $\mathbf{Z}$ is a
minimal $A-Y$ cut in $\mathcal{H}^{1}$.

\item $\mathbf{Z}$ is a minimum $\mathbf{L}-\mathbf{N}$ static adjustment
set with respect to $A,Y$ in $\mathcal{G}$ if and only if $\mathbf{Z}$ is a
minimum $A-Y$ cut in $\mathcal{H}^{1}$.
\end{enumerate}
\end{lemma}

\begin{proof}[Proof of Lemma \ref{lemma:equiv}]

Our proof will make use of the following facts.
By Theorem 1 and Corollary 2 of \cite{vanAI}, if $\mathbf{Z}$ is a minimal $\mathbf{L}-\mathbf{N}$ static adjustment set then $\mathbf{Z}\subset \an_{\mathcal{G}^{pbd}}(\lbrace A, Y\rbrace\cup \mathbf{L})$. Note also that $ \an_{\mathcal{G}}(\lbrace A, Y\rbrace \cup \mathbf{L})=\an_{\mathcal{G}^{pbd}(A,Y)}(\lbrace A, Y\rbrace \cup \mathbf{L})$.

1) By assumption $Y \perp_{\mathcal{H}^{1}} A \mid \mathbf{Z}$. Lemma \ref{lemma:cuts_contain} implies 
$\mathbf{L}\subset \mathbf{Z} \subset \an_{\mathcal{G}}(\lbrace A,Y\rbrace \cup \mathbf{L})\cap \mathbf{N}$ and $\mathbf{Z}\cap \forb(A,Y,\mathcal{G}) = \emptyset$. In particular, since $\mathbf{L}\subset \mathbf{Z}$, Lemma \ref{lemma:indep_equiv} implies $Y \perp_{\mathcal{H}^{0}} A \mid \mathbf{Z}$.
Let $\mathbf{M}^{\prime}=\an_{\mathcal{G}}(\lbrace A,Y\rbrace \cup \mathbf{Z})$ and $\mathcal{H}^{0, \prime}= \left\lbrace \mathcal{G}^{pbd}_{\mathbf{M}^{\prime}}(A,Y)\right\rbrace^{m}$. Since $\mathbf{Z} \subset \an_{\mathcal{G}}(\lbrace A,Y\rbrace \cup \mathbf{L})$ we have that
$\mathbf{M}^{\prime} \subset \an_{\mathcal{G}}(\lbrace A,Y\rbrace \cup \mathbf{L})$. This implies that $\mathcal{G}^{pbd}_{\mathbf{M}^{\prime}}(A,Y)$ is a subgraph of $\mathcal{G}^{pbd}_{\an_{\mathcal{G}}(\lbrace A,Y\rbrace \cup \mathbf{L})}(A,Y)$ and hence that $\mathcal{H}^{0, \prime}$ is a subgraph of $\mathcal{H}^{0}$. Then $Y \perp_{\mathcal{H}^{0}} A \mid \mathbf{Z}$ implies $Y \perp_{\mathcal{H}^{0,\prime}} A \mid \mathbf{Z}$. Now \eqref{eq:moral_equiv} implies $Y \ort_{\mathcal{G}^{pbd}(A,Y)} A \mid \mathbf{Z}$. Since moreover $\mathbf{Z}\cap \forb(A,Y,\mathcal{G}) = \emptyset$ and $\mathbf{L}\subset \mathbf{Z}\subset \mathbf{N}$, Theorem \ref{theo:van} implies that $\mathbf{Z}$ is a $\mathbf{L}-\mathbf{N}$ static adjustment set with respect to $A,Y$ in $\mathcal{G}$.

2) By part 1) we only need to prove that if $\mathbf{Z}$ is an $\mathbf{L}-\mathbf{N}$ static adjustment set with respect to $A,Y$ in $\mathcal{G}$ and $\mathbf{Z} \subset \an_{\mathcal{G}}(\lbrace A,Y\rbrace \cup \mathbf{L})$  then $\mathbf{Z}$ is an $A-Y$ cut in $\mathcal{H}^{1}$. Assume that $\mathbf{Z}$ is an $\mathbf{L}-\mathbf{N}$ static adjustment set and $\mathbf{Z} \subset \an_{\mathcal{G}}(\lbrace A,Y\rbrace \cup \mathbf{L})$. Then by Theorem 1, $Y \ort_{\mathcal{G}^{pbd}(A,Y)} A \mid \mathbf{Z}$ and $\mathbf{Z}\cap \forb(A,Y,\mathcal{G}) = \emptyset$. Moreover, since $\mathbf{L}\subset\mathbf{Z}\subset \an_{\mathcal{G}}(\lbrace A,Y\rbrace \cup \mathbf{L})$ we have that $\an_{\mathcal{G}}(\lbrace A,Y\rbrace\cup \mathbf{Z})=\an_{\mathcal{G}}(\lbrace A,Y\rbrace \cup \mathbf{L})$. Thus,  $Y \ort_{\mathcal{G}^{pbd}(A,Y)} A \mid \mathbf{Z}$ and \eqref{eq:moral_equiv} imply that $\mathbf{Z}$ is an $A-Y$ cut in $\mathcal{H}^{0}$. Since $\mathbf{L}\subset\mathbf{Z}\subset \an_{\mathcal{G}}(\lbrace A,Y\rbrace \cup \mathbf{L}) \cap \mathbf{N}$ and $\mathbf{Z}\cap \forb(A,Y,\mathcal{G}) = \emptyset$, Lemma \ref{lemma:indep_equiv} implies that $\mathbf{Z}$ is an $A-Y$ cut in $\mathcal{H}^{1}$, which is what we wanted to show.

3) Take $\mathbf{Z}$ a minimal $\mathbf{L}-\mathbf{N}$ static adjustment set.  We will show that $\mathbf{Z}$ is a minimal $A-Y$ cut in $\mathcal{H}^{1}$.
Since $\mathbf{L}\subset\mathbf{Z}\subset \an_{\mathcal{G}}(\lbrace A,Y\rbrace \cup \mathbf{L})\cap \mathbf{N}$, part 2) implies that
$\mathbf{Z}$ is an $A-Y$ cut in $\mathcal{H}^{1}$. Assume for the sake of contradiction that $\mathbf{Z}$ is a not a minimal $A-Y$ cut in $\mathcal{H}^{1}$. Then there exists $\mathbf{Z}^{\prime} \subsetneq \mathbf{Z}$ that is an $A-Y$ cut in $\mathcal{H}^{1}$. By part 1), $\mathbf{Z}^{\prime}$ is an $\mathbf{L}-\mathbf{N}$ static adjustment set with respect to $A,Y$ in $\mathcal{G}$, which contradicts the fact that $\mathbf{Z}$ was a minimal $\mathbf{L}-\mathbf{N}$ static adjustment set. 
 
 Now take $\mathbf{Z}$ a minimal $A-Y$ cut in $\mathcal{H}^{1}$. We will show that $\mathbf{Z}$ is a minimal $\mathbf{L}-\mathbf{N}$ static adjustment set with respect to $A,Y$ in $\mathcal{G}$. By part 1), we know that $\mathbf{Z}$ is an $\mathbf{L}-\mathbf{N}$ static adjustment set with respect to $A,Y$ in $\mathcal{G}$. Assume for the sake of contradiction that $\mathbf{Z}$ is not a minimal $\mathbf{L}-\mathbf{N}$ static adjustment set. {Then there exists $\mathbf{Z}^{\prime} \subsetneq \mathbf{Z}$ that is a minimal $\mathbf{L}-\mathbf{N}$ static adjustment set.} Arguing as before, we see that $\mathbf{Z}^{\prime}$ is an $A-Y$ cut in $\mathcal{H}^{1}$. This contradicts the fact that $\mathbf{Z}$ was a minimal $A-Y$ cut in $\mathcal{H}^{1}$

4) This can be proven using arguments analogous to those used in the proof of part 3).
\end{proof}

We are now ready to prove Proposition \ref{prop:equiv_dyn}.

\begin{proof}[Proof of Proposition \ref{prop:equiv_dyn}]
We prove parts 1) and 2). Parts 3) - 5) follow immediately from Lemma \ref{lemma:equiv} and Proposition \ref{prop:equiv_dtr}.

1) 
Since $(\mathbf{L, N})$ form an admissible pair with respect to $A,Y$ in $\mathcal{G}$, by Theorem 2 of \cite{vanAI} $\mathbf
{W}=\left\lbrace\an_{\mathcal{G}}(\lbrace A, Y\rbrace \cup \mathbf{L})\cap \mathbf{N}\right\rbrace \setminus \forb(A,Y,\mathcal{G})$ is an $\mathbf{L}-\mathbf{N}$ static adjustment set. Since $\mathbf{W}\subset \an_{\mathcal{G}}(\lbrace A, Y\rbrace \cup \mathbf{L}) $, by part 2) of Lemma \ref{lemma:equiv}, $\mathbf{W}$ is an $A-Y$ cut in $\mathcal{H}^{1}$. Hence $A$ and $Y$ cannot be adjacent in $\mathcal{H}^{1}$.

2)
Let $\mathbf{Z}$ be an $A-Y$ cut in $\mathcal{H}^{1}$.
By part 1) of Lemma \ref{lemma:equiv}, $\mathbf{Z}$ is a $\mathbf{L}-\mathbf{N}$ static adjustment set with respect to $A,Y$ in $\mathcal{G}$. 
Then part 1) of Proposition \ref{prop:equiv_dtr} implies that $\mathbf{Z}$ is an $\mathbf{L}-\mathbf{N}$ dynamic adjustment set with respect to $A,Y$ in $\mathcal{G}$.
\end{proof}

\begin{proof}[Proof of Proposition \ref{prop:order_implies_indep}]
We begin with the proof of \eqref{eq:indep_opt_y}. Since $\mathbf{Z}_{1} \unlhd_{\mathcal{H}^{1}} \mathbf{Z}_{2}$, we have that that $Y \perp_{\mathcal{H}^{1}} \mathbf{Z}_{2}\setminus \mathbf{Z}_{1} \mid  \mathbf{Z}_{1}$. Lemma \ref{lemma:indep_equiv} implies $Y \perp_{\mathcal{H}^{0}} \mathbf{Z}_{2}\setminus \mathbf{Z}_{1} \mid  \mathbf{Z}_{1}$ and hence
\begin{equation}
Y \perp_{\mathcal{H}^{0}} \mathbf{Z}_{2}\setminus \mathbf{Z}_{1} \mid  \mathbf{Z}_{1}, A. 
\label{eq:indep_Y_H0}
\end{equation}
Recall that 
$\mathcal{H}^{0}=\left\lbrace \mathcal{G}^{pbd}_{\an_{ \mathcal{G}}(\lbrace A, Y\rbrace  \cup \mathbf{L})}(A,Y)\right\rbrace^{m}$.
Also,  $\mathbf{Z}_{1}, \mathbf{Z}_{2}\subset \mathbf{V}\left(\mathcal{H}^{1} \right) \subset \an_{ \mathcal{G}}(\lbrace A, Y\rbrace \cup \mathbf{L})=\an_{ \mathcal{G}^{pbd}(A,Y)}(\lbrace A, Y\rbrace \cup \mathbf{L})$ and $\mathbf{L}\subset \mathbf{Z}_{1}$, $\mathbf{L}\subset \mathbf{Z}_{2}$. Then $\an_{\mathcal{G}}(\lbrace A,Y\rbrace \cup \mathbf{Z}_{1} \cup \mathbf{Z}_{2})=\an_{\mathcal{G}}(\lbrace A,Y\rbrace \cup \mathbf{L})$.
Hence equations \eqref{eq:moral_equiv} and \eqref{eq:indep_Y_H0}  imply
\begin{equation}
Y \ort_{\mathcal{G}^{pbd}(A,Y)} \mathbf{Z}_{2}\setminus \mathbf{Z}_{1} \mid  \mathbf{Z}_{1}, A.
\label{eq:indep_Y_pbd}
\end{equation}

Now, assume for the sake of contradiction that \eqref{eq:indep_opt_y} does not hold, and hence that there exists a path $\pi$ in $\mathcal{G}$ between $Y$ and a vertex $Z \in  \mathbf{Z}_{2}\setminus \mathbf{Z}_{1}$ that is open in $\mathcal{G}$ given $A, \mathbf{Z}_{1}$. Assume first that $\pi$ does not have colliders, and hence it is either directed or has a single fork. Since $\pi$ is open in $\mathcal{G}$ given $A, \mathbf{Z}_{1}$, $\pi$ does not intersect $A$. Since the proper back-door graph $\mathcal{G}^{pbd}(A,Y)$ is formed by removing from $\mathcal{G}$ the first edge in all causal paths between $A$ and $Y$, the path $\pi$ must exist in $\mathcal{G}^{pbd}(A,Y)$. This contradicts \eqref{eq:indep_Y_pbd}. Hence $\pi$ has to have at least one collider. Since $\pi$ is open in $\mathcal{G}$ given $A, \mathbf{Z}_{1}$, all colliders in $\pi$ must be ancestors of a vertex in $\lbrace A \rbrace \cup \mathbf{Z}_{1}$ an no non-collider in $\pi$ can be in $\lbrace A \rbrace \cup \mathbf{Z}_{1}$. Again, the definition of the proper back-door graph implies that $\pi$ must exist in $\mathcal{G}^{pbd}(A,Y)$ and that all colliders in $\pi$ are also ancestors in $\mathcal{G}^{pbd}(A,Y)$ of a vertex in $\lbrace A \rbrace \cup \mathbf{Z}_{1}$. This contradicts \eqref{eq:indep_Y_pbd}.
It must be that \eqref{eq:indep_opt_y} holds.

Turn now to the proof of \eqref{eq:indep_opt_a}. Since $\mathbf{Z}_{1} \unlhd_{\mathcal{H}^{1}} \mathbf{Z}_{2}$ we have that
$A \perp_{\mathcal{H}^{1}}\mathbf{Z}_{1}\setminus \mathbf{Z}_{2} \mid \mathbf{Z}_{2}$.  By Lemma \ref{lemma:indep_equiv}, this implies that
\begin{equation}
A \perp_{\mathcal{H}^{0}}\mathbf{Z}_{1}\setminus \mathbf{Z}_{2} \mid \mathbf{Z}_{2}.
\label{eq:indep_A_H0}
\end{equation}
Let 
$\widetilde{\mathbf{M}}(A,Y,\mathcal{G})= \an_{ \mathcal{G}^{pbd}(A,Y)}(\lbrace A \rbrace \cup \mathbf{Z}_{1}\cup\mathbf{Z}_{2})$ and
$\widetilde{\mathcal{H}}^{0}(A,Y,\mathcal{G})=\left\lbrace \mathcal{G}^{pbd}_{\widetilde{\mathbf{M}}(A,Y,\mathcal{G})}(A,Y)\right\rbrace^{m}$.
We will show that \eqref{eq:indep_A_H0} implies
\begin{equation}
A \perp_{\widetilde{\mathcal{H}}^{0}(A,Y,\mathcal{G})}  \mathbf{Z}_{1}\setminus \mathbf{Z}_{2} \mid \mathbf{Z}_{2}.
\label{eq:indep_A_H0_tilde}
\end{equation}
Note that, since $\mathbf{Z}_{1}, \mathbf{Z}_{2}\subset  \an_{ \mathcal{G}}(\lbrace A, Y\rbrace \cup \mathbf{L})$, we have $\widetilde{\mathbf{M}}(A,Y,\mathcal{G}) \subset \an_{ \mathcal{G}}(\lbrace A, Y\rbrace \cup \mathbf{L})$. {Thus $\mathcal{G}^{pbd}_{\widetilde{\mathbf{M}}(A,Y,\mathcal{G})}(A,Y)$ is a subgraph of $\mathcal{G}^{pbd}_{\an_{ \mathcal{G}}(\lbrace A, Y\rbrace \cup \mathbf{L})}(A,Y)$ and 
$\widetilde{\mathcal{H}}^{0}(A,Y,\mathcal{G})$ is a subgraph of ${\mathcal{H}}^{0}(A,Y,\mathcal{G})$}. Hence, \eqref{eq:indep_A_H0_tilde} follows from \eqref{eq:indep_A_H0}.
Now equations \eqref{eq:moral_equiv} and \eqref{eq:indep_A_H0_tilde}  imply
\begin{equation}
A \ort_{\mathcal{G}^{pbd}(A,Y)}   \mathbf{Z}_{1}\setminus \mathbf{Z}_{2} \mid \mathbf{Z}_{2}.
\label{eq:indep_A_pbd}
\end{equation}

Next, assume for the sake of contradiction that \eqref{eq:indep_opt_a} does not hold, and hence that there exists a path $\pi$ in $\mathcal{G}$ between $A$ and a vertex $Z \in   \mathbf{Z}_{1}\setminus \mathbf{Z}_{2}$ that is open in $\mathcal{G}$ given $\mathbf{Z}_{2}$. Assume first that $\pi$ does not have colliders, and hence it is either directed or has a single fork. Since $\pi$ is open in  $\mathcal{G}$ given $\mathbf{Z}_{2}$, $\pi$ does not intersect $\mathbf{Z}_{2}$. Since $\mathbf{Z}_{1} \cap \forb(A,Y,\mathcal{G})=\emptyset$, $\mathbf{Z}_{1}\subset \an_{\mathcal{G}}(\lbrace A,Y\rbrace \cup \mathbf{L})$ and $\mathbf{L}\cap \de_{\mathcal{G}}(A) = \emptyset$ no edge in $\pi$ can be of the form $A \rightarrow V$ for some $V$ in $\pi$.  Since the proper back-door graph $\mathcal{G}^{pbd}(A,Y)$ is formed by removing from $\mathcal{G}$ the first edge in all causal paths from $A$ to $Y$, the path $\pi$ must exist in $\mathcal{G}^{pbd}(A,Y)$. This contradicts \eqref{eq:indep_A_pbd}.
Hence $\pi$ has to have at least one collider. Since $\pi$ is open in $\mathcal{G}$ given $\mathbf{Z}_{2}$, all colliders in $\pi$ must be ancestors in $\mathcal{G}$ of a vertex in $\mathbf{Z}_{2}$ an no non-collider in $\pi$ can be in $\mathbf{Z}_{2}$. We can assume without loss of generality that $A$ only appears once on the path $\pi$. Then the only edge in $\pi$ that could possibly not be an edge in $\mathcal{G}^{pbd}$ is the edge that contains $A$, if it points out of $A$. But if the edge points out of $A$, the collider on $\pi$ that is closest to $A$ would be a descendant of $A$, and hence could not be an ancestor of a vertex in $\mathbf{Z}_{2}$, which is a contradiction. Thus $\pi$ must exist in $\mathcal{G}^{pbd}(A,Y)$.
  Since $\mathbf{Z}_{2} \cap \forb(A,Y,\mathcal{G})=\emptyset$, $\mathbf{Z}_{2}\subset \an_{\mathcal{G}}(\lbrace A,Y\rbrace \cup \mathbf{L})$ and $\mathbf{L}\cap \de_{\mathcal{G}}(A) = \emptyset$, all colliders in $\pi$ are also ancestors in $\mathcal{G}^{pbd}(A,Y)$ of a vertex in $\mathbf{Z}_{2}$. This contradicts \eqref{eq:indep_A_pbd}.
It must be that \eqref{eq:indep_opt_a} holds. This finishes the proof of the proposition.
\end{proof}

\begin{proof}[Proof of Theorem \ref{theo:main_indep}]

We begin with the proof of part 1). If $\mathbf{N}=\mathbf{V}$, the result follows from Proposition \ref{prop:opt_adj}. Assume then that $\mathbf{N}\subset \an_{\mathcal{G}}(\lbrace A, Y\rbrace \cup \mathbf{L})$. By part 1) of Proposition \ref{prop:equiv_dyn}, $A$ and $Y$ are not adjacent in $\mathcal{H}^{1}$. Thus, any path in $\mathcal{H}^{1}$ from $A$ to $Y$ has to intersect $\mathbf{O}=\nb_{\mathcal{H}^{1}}(Y)$. It follows that $\mathbf{O}$ is an $A-Y$ cut in $\mathcal{H}^{1}$. Moreover, it is easy to show that $\mathbf{O}\unlhd_{\mathcal{H}^{1}}  \mathbf{Z}$ for any other $A-Y$ cut $\mathbf{Z}$. Since $\mathbf{N}\subset \an_{\mathcal{G}}(\lbrace A, Y\rbrace \cup \mathbf{L})$, by part 3) of Proposition \ref{prop:equiv_dyn}, $\mathbf{Z}$ is an $\mathbf{L}-\mathbf{N}$ dynamic adjustment set with respect to $A,Y$ in $\mathcal{G}$ if and only if $\mathbf{Z}$ is an $A-Y$ cut in $\mathcal{H}^{1}$. The desired result now follows from Proposition \ref{prop:order_implies_indep}.

Turn now to the proof of parts 2) and 3). Since by part 1) of Proposition \ref{prop:equiv_dyn}, $A$ and $Y$ are not adjacent in $\mathcal{H}^{1}$,  Theorems 1 and 2 from \cite{lattice} imply that $\mathbf{O}_{min}$ and $\mathbf{O}_{m}$ are $A-Y$ cuts in $\mathcal{H}^{1}$ and, moreover, $\mathbf{O}_{min}\unlhd_{\mathcal{H}^{1}}  \mathbf{Z} $ for all $\mathbf{Z}$ that is a minimal $A-Y$ cut in $\mathcal{H}^{1}$ and $\mathbf{O}_{m}\unlhd_{\mathcal{H}^{1}}  \mathbf{Z} $ for all $\mathbf{Z}$ that is a minimum $A-Y$ cut in $\mathcal{H}^{1}$. By parts 4) and 5) of Proposition \ref{prop:equiv_dyn}, the set of minimal (minimum) $\mathbf{L}-\mathbf{N}$ dynamic adjustment sets with respect to $A,Y$ in $\mathcal{G}$ is equal to the set of minimal (minimum)
$A-Y$ cuts in $\mathcal{H}^{1}$. The desired result now follows from Proposition \ref{prop:order_implies_indep}.
\end{proof}

\subsection{Proofs of results in Section \protect\ref{sec:algos}}

\begin{proof}[Proof of Lemma \ref{lemma:isinmin_correct}]
First note that the new graph $\mathcal H'$ is constructed from the graph $\mathcal H$ by adding at most two edges: $A - V$ and $V - Y$. Assume that the algorithm return $\true$. Then $m_1 = m_2$. We will show that $V$ is included in a minimum $A - Y$ cut for $\mathcal{H}$. Since $A - V - Y$ is a path in $\mathcal H'$, any minimum $A - Y$ cut in $\mathcal H'$ has to include $V$. Let $\mathbf{Z}^{\prime}$ be a minimum $A - Y$ cut in $\mathcal H'$. Then $ \# \mathbf{Z}^{\prime} =m_{1}$. Since $\mathcal{H}$ is a sub-graph of $\mathcal{H}^{\prime}$, $\mathbf{Z}^{\prime}$ is also an $A-Y$ cut in $\mathcal{H}$ and hence $m_{2} \leq \# \mathbf{Z}^{\prime} =m_{1}$. Since by assumption $m_1 = m_2$, it must be that $\# \mathbf{Z}^{\prime} = m_{2}$. Hence $\mathbf{Z}^{\prime}$ is a minimum $A-Y$ cut in $\mathcal{H}$ that satisfies $V \in \mathbf{Z}^{\prime}$.

Assume next that there exists a minimum $A - Y$ cut $\mathbf Z$ in $\mathcal H$ such that $V \in \mathbf Z$. Then $\#\mathbf{Z}=m_{2}$. Clearly $\mathbf Z$ is also an $A-Y$ cut in $\mathcal{H}^{\prime}$. This implies $m_{2} = \#\mathbf{Z} \geq m_{1}$.  However, since $\mathcal H$ is a sub-graph of $\mathcal H'$, we have $m_1 \geq m_2$. Hence $m_1 = m_2$ and the algorithm outputs $\true$. 
\end{proof}

\begin{proof}[Proof of Proposition \ref{prop:algo_main_correct}]
By assumption \testExistsAdj{$A, Y, \mathcal{G}, \mathbf{N},\mathbf{L}$} $= \true$. Let us call the output of the algorithm $\mathbf Z^\star$.  By Manger's Theorem (see for example Chapter 7 of \cite{graph-book}), the output of \texttt{minCut}$(A, Y , \mathcal H)$ coincides with the number of paths returned by \texttt{disjointPaths}$(A, Y, \mathcal{H})$. Let $\pi_{1},\dots, \pi_{m}$ be the paths returned by \texttt{disjointPaths}$(A, Y, \mathcal{H})$. 

Now, since $\mathbf{O}_{m}$ is an $A-Y$ cut in $\mathcal{H}^{1}$ of size $m$, there is exactly one vertex $V_j  \in \mathbf{O}_{m}$ in each path $\pi_j$, $j = 1, \ldots, m$. The definition of $\mathbf{O}_{m}$ implies that such $V_{j}$ is the vertex on $\pi_{j}$ that: (i)  is a member of at least one minimum $A-Y$ cut, and (ii) is closer to $Y$ on $\pi_{j}$ than any other vertex on $\pi_{j}$ that is a member of at least one minimum $A-Y$ cut. These are precisely the vertices that are included in  $\mathbf Z^\star$. Thus  $\mathbf Z^\star =  \mathbf{O}_{m}$.

Next, we will bound the worst case complexity of Algorithm \ref%
{algo:find_opt_min}. To do so, we first need to bound the cardinalities of $%
\mathbf{V}(\mathcal{H}^{1})$ and of $\mathbf{E}(\mathcal{H}^{1})$ as a
function of $\# \mathbf{V}(\mathcal{G})$. Clearly $\# \mathbf{V}(\mathcal{H}%
^{1}) \leq \# \mathbf{V}(\mathcal{G}) $. 
This in turn implies $\# \mathbf{E}(\mathcal{H}^{1}) \leq \lbrace \#\mathbf{V%
}(\mathcal{G}) \rbrace^{2}$. Now, the first step in Algorithm \ref%
{algo:find_opt_min} is running \texttt{testExistAdj}($A, Y, \mathbf{L}, 
\mathbf{N},\mathcal{G}$), which has complexity $\mathcal{O}\lbrace \# 
\mathbf{V}(\mathcal{G}) + \# \mathbf{E}(\mathcal{G})\rbrace=\mathcal{O}\left[
\lbrace\# \mathbf{V}(\mathcal{G})\rbrace^{2}\right]$. Next, the algorithm
constructs $\mathcal{H}^{1}$. Since \cite{vanAI} show that the proper
back-door graph can be constructed in $\mathcal{O}\left[\lbrace \# \mathbf{V}%
(\mathcal{G})\rbrace^{2}\right]$ time, that $\an_{\mathcal{G}}(\lbrace
A,Y\rbrace \cup \mathbf{L})$ and $\forb(A,Y,\mathcal{G})$ can be computed in 
$\mathcal{O}\left[\lbrace \# \mathbf{V}(\mathcal{G})\rbrace^{2}\right]$ time
(see their Section 6 for these three claims) and that a graph $\mathcal{G}$
can be moralized in $\mathcal{O}\left[\lbrace \# \mathbf{V}(\mathcal{G}%
)\rbrace^{2}\right]$ time (see their Lemma 2), we have that the construction
of $\mathcal{H}^{1}$ has complexity $\mathcal{O}\left[\lbrace \# \mathbf{V}(%
\mathcal{G})\rbrace^{2}\right]$. Next, Algorithm \ref{algo:find_opt_min}
makes one call to \texttt{disjointPaths}($A, Y, \mathcal{H}^{1}$), which has
complexity $\mathcal{O}\left[ \lbrace \# \mathbf{V}(\mathcal{H}^{1})
\rbrace^{1/2} \# \mathbf{E}\lbrace \mathcal{H}^{1}\rbrace \right]=\mathcal{O}%
\left[\lbrace \# \mathbf{V}(\mathcal{G})\rbrace^{5/2}\right]$. After that,
the algorithm makes at most $\# \mathbf{V}(\mathcal{G})$ calls to %
\isInMinimum{$V_j, \mathcal H_1$}. The complexity of 
\isInMinimum{$V_j,
\mathcal H_1$} is bounded by the complexity of \texttt{minCut}($A, Y, 
\mathcal{H}_{1}$), which is bounded by $\mathcal{O}\left[\lbrace \# \mathbf{V%
}(\mathcal{G})\rbrace^{5/2}\right]$. Hence, the overall complexity of the
outer for loop is $\mathcal{O}\left[\lbrace \# \mathbf{V}(\mathcal{G}%
)\rbrace^{7/2}\right]$. This dominates the complexities of all other steps.
We conclude that the worst case complexity of Algorithm \ref%
{algo:find_opt_min} is $\mathcal{O}\left[\lbrace \# \mathbf{V}(\mathcal{G}%
)\rbrace^{7/2}\right]$.
\end{proof}

\begin{proof}[Proof of Proposition \ref{prop:algo_minimal}]
Since by assumption \testExistsAdj{$A, Y, \mathcal{G}, \mathbf{N},\mathbf{L}$} $= \true$, it follows from part 1) of Proposition \ref{prop:equiv_dyn} that $A$ and $Y$ are not adjacent in $\mathcal{H}^{1}$. We claim that, at any iteration of the algorithm,
\begin{align}
&\text{if }V\in\visited\text{, there exists a path in }\mathcal{H}^{1} \text{ from  }V\text{ to }A\text{ that does not intersect} \nb_{\mathcal{H}^{1}}(Y)\text{ except possibly at }V.
\label{eq:induct}
\end{align}
 We prove this by induction on the number of times the while loop was entered, say $k$. Before entering the while loop for the first time, $\visited=\emptyset$ and hence \eqref{eq:induct} holds trivially. Suppose that after $k\geq 0$ iterations of the while loop \eqref{eq:induct} holds. Take $V$ a vertex that is a member of $\visited$ after $k+1$ iterations. If $V$ was already a vertex in $\visited$ after $k$ iterations, then by the inductive assumption, there is a path in $\mathcal{H}^{1}$ from $V$ to $A$ that does not intersect $\nb_{\mathcal{H}^{1}}(Y)$, except possibly at $V$. If $V$ was only added to $\visited$ after $k+1$ iterations, then $V$ is a neighbor of a vertex, say $W$, that was a already a member of $\visited$ after $k$ iterations, but not a member of $\nb_{\mathcal{H}^{1}}(Y)$. By the inductive hypothesis, there is a path in $\mathcal{H}^{1}$ from $W$ to $A$ that does not intersect $\nb_{\mathcal{H}^{1}}(Y)$. Since $W$ and $V$ are adjacent, we conclude that there exists a path $V$ to $A$ that does not intersect $\nb_{\mathcal{H}^{1}}(Y)$, except possibly at $V$. This finishes the proof that \eqref{eq:induct} holds at any iteration of the algorithm.
 
Now note that at any iteration, $\out$ is formed by the vertices in $\visited$ that are adjacent to $Y$. Then the fact that \eqref{eq:induct} holds implies that $\out \subset \mathbf{O}_{min}$. So to prove the proposition, it suffices to show that when the algorithm finishes, $\mathbf{O}_{min} \subset \out$. Take a vertex $O \in \mathbf{O}_{min}$. Then in $\mathcal{H}^{1}$ there is a path, say $\pi$, from $O$ to $A$ that only intersects $\nb_{\mathcal{H}^{1}}(Y)$ at $O$. The vertex adjacent to $A$ in $\pi$ is added to the stack during the first iteration of the while loop. During the next iterations, all subsequent vertices in $\pi$ are visited and their neighbours added to the stack, until $O$ is reached. When $O$ is reached, since it is a neighbor of $Y$, it is added to $\out$. Thus, when the algorithm finishes, $\mathbf{O}_{min} \subset \out$.

This finishes the proof of the proposition.
\end{proof}

\FloatBarrier
\bibliographystyle{apalike}
\bibliography{efficient}

\end{document}